\documentclass{amsart}

\usepackage{amsmath,amssymb,enumerate}
\usepackage{mathptmx}

\usepackage[all]{xy}
\usepackage{amsthm}
\usepackage{amssymb} 
\usepackage{amsmath,amscd} 
\usepackage{mathrsfs}
\usepackage{color}
\usepackage{enumerate}

\def\GrMod{\operatorname{\mathsf{GrMod}}}

\def\turn!{\textup{!`}}

\def\op{\textup{op}}

\def\pd{\mathop{\mathrm{pd}}\nolimits}

\def\injdim{\mathop{\mathrm{id}}\nolimits}

\def\mrb{\mathrm{b}}

\def\Spec{\operatorname{Spec}}

\def\thick{\mathop{\mathsf{thick}}\nolimits}

\def\Tor{\operatorname{Tor}}

\def\kk{{\mathbf k}}

\def\NN{{\mathbb N}}

\def\ZZ{{\mathbb Z}}

\def\cD{{\mathcal D}}
\def\cE{{\mathcal E}}
\def\cF{{\mathcal F}}

\def\cI{{\mathcal I}}

\def\cP{{\mathcal P}}

\def\cX{{\mathcal X}}
\def\cY{{\mathcal Y}}
\def\cZ{{\mathcal Z}}

\def\sfC{{\mathsf{C}}}
\def\sfD{{\mathsf{D}}}

\def\sfH{{\mathsf{H}}}

\def\sfK{{\mathsf{K}}}

\def\tuB{{\textup{B}}}

\def\tuH{{\textup{H}}}

\def\tuZ{{\textup{Z}}}

\def\frkp{{\mathfrak{p}}}

\def\id{\operatorname{id}}
\def\op{\operatorname{op}}

\def\mod{\operatorname{mod}}
\def\Mod{\operatorname{Mod}}
\def\GrMod{\operatorname{Mod}^{\mathbb{Z}}}

\def\Ker{\mathop{\mathrm{Ker}}\nolimits}

\def\Proj{\operatorname{Proj}} 
\def\Inj{\operatorname{Inj}}

\def\Add{\operatorname{Add}}
\def\add{\operatorname{add}}

\def\Cok{\operatorname{Cok}}
\def\Coker{\operatorname{Cok}}
\def\image{\operatorname{Im}}

\def\Hom{\operatorname{Hom}}

\def\End{\operatorname{End}}

\def\Ext{\operatorname{Ext}}

\def\gldim{\operatorname{gldim}}

\newcommand{\RHom}{\operatorname{\Bbb{R}Hom}}

\newcommand{\lotimes}{\otimes^{\mathbb{L}}}

\newcommand{\cone}{\operatorname{\mathsf{cn}}}

\def\RHom{\operatorname{\mathbb{R}Hom}}

\def\Perf{\mathsf{perf}}

\newtheorem{lemma}{Lemma}[section]
\newtheorem{proposition}[lemma]{Proposition}
\newtheorem{theorem}[lemma]{Theorem}
\newtheorem{corollary}[lemma]{Corollary}

\newtheorem{claim}[lemma]{Claim}

\theoremstyle{definition}

\newtheorem{remark}[lemma]{Remark}
\newtheorem{example}[lemma]{Example}
\newtheorem{definition}[lemma]{Definition}
\newtheorem{conjecture}[lemma]{Conjecture}

\theoremstyle{remark}

\def\champ{{\mathop{\mathrm{amp}}}}
\def\depth{{\mathop{\mathrm{depth}}}}

\def\inf{{\mathop{\mathrm{inf}}}}
\def\sup{{\mathop{\mathrm{sup}}}}

\def\ulfrkp{{\underline{\frkp}}}

\def\fd{\mathop{\mathrm{fd}}\limits}

\title[Resolutions  of DG-modules]
{
Resolutions and homological dimensions of DG-modules
}

\date{}

\dedicatory{\small Dedicated to L. Avramov's 70 birthday}
\author{Hiroyuki Minamoto
}



\begin{document}

\address{Department of Mathematics Osaka Prefecture University,\\ Sakai City, Japan}
\email{minamoto@mi.s.osakafu-u.ac.jp}

\begin{abstract}
Recently, Yekutieli introduced projective dimension, injective dimension and flat dimension of DG-modules 
by generalizing the characterization of projective dimension, injective dimension  and flat dimension of ordinary modules 
by vanishing of $\Ext$ or $\Tor$-groups.  
In  this paper, 
we introduce DG-version of projective, injective and flat resolution for  DG-modules over a connective DG-algebra  
which are different from known DG-version of projective,  injective and flat  resolutions. 
An important feature of these resolutions is  that, roughly speaking, the ``length" of these resolutions give projective, injective or flat dimensions. 
We show that these resolutions allows us to investigate basic properties of projective, injective and flat dimensions of DG-modules. 
As an application we introduce the global dimension of a connective DG-algebra 
and show that finiteness of global dimension is derived invariant. 
\end{abstract}

\maketitle

\tableofcontents

\section{Introduction}
Differential graded (DG) algebra lies in the center of homological algebra 
and allows us to use techniques of homological algebra of  ordinary algebras in much wider context.  
The projective resolutions and the injective resolutions which are the fundamental tools of homological algebra 
already have their DG-versions, which are called a DG-projective resolution and a DG-injective resolution.  
The aim of  this paper is to introduce a different 
DG-versions  
for  DG-modules over a connective DG-algebra. 
The motivation came from the projective dimensions and the injective dimensions for DG-modules introduced by Yekutieli. 
 
We explain the details by focusing on the projective dimension and the projective resolution. 
Let $R$ be an ordinary algebra. 
One of the most fundamental and basic homological invariant for a (right) $R$-module $M$ is the projective dimension $\pd_{R}M$. 
Avramov-Foxby \cite{Avramov-Foxby:Homological dimensions}  generalized the projective dimension 
for an object of the derived category $M \in \sfD(R)$. 
Recently, Yekutieli \cite{Yekutieli} introduced the projective dimension $\pd_{R} M$ 
for an object of $M \in \sfD(R)$ in the case where $R$ is a DG-algebra 
from the view point that 
the number  $\pd_{R} M$ measures how the functor $\RHom_{R}(M, -)$ changes the amplitude of the cohomology groups. 
(In the case where $R$ is an ordinary ring, for a DG-$R$-module $M \in \sfC(R)$  we have two projective dimensions, by Avramov-Foxby  and by Yekutieli. 
 It is explained in Remark \ref{AF projective dimension} that  they are essentially the same.)

Let $R$ be an ordinary ring and $M$ an $R$-module, again. 
Recall that the projective dimension $\pd_{R} M$ is characterized as the smallest length of 
projective resolutions $P_{\bullet}$ 
\[
 0 \to P_{d} \to P_{d-1} \to \cdots \to P_{1} \to P_{0} \to M \to 0. 
\]
There are the notion of DG-projective resolution, which is also called projectively cofibrant replacement  and so on,
which is a generalization of  a projective resolution for a DG-module $M$ over a  DG-algebra $R$. 
However, it is not suitable to measure the projective dimension.  
The aim of this paper is to introduce a notion of a sup-projective (sppj) resolution of an object of $M \in \sfD(R)$ 
which can measure the projective dimension, in the case where $R$ is a connective ($=$non-positive) DG-algebra. 

Recall that a (cohomological) DG-algebra $R$ is called \textit{connective} if 
the vanishing condition $\tuH^{> 0}(R) =0$ of the cohomology groups is satisfied. 
There are rich sources  of connective DG-algebras:  
the Koszul algebra $K_{R}(x_{1}, \cdots, x_{d})$ in  commutative ring theory, 
and an endomorphism DG-algebra $\RHom(S,S)$ of a silting object $S$ 
(for silting object see \cite{Aihara-Iyama}). 
We would like to point out that a commutative connective DG-algebras 
are regarded as the coordinate algebras of derived affine schemes 
in derived algebraic geometry (see e.g. \cite{Gaitsgory-Rozenblyum}).

Let $R$ be a connective DG-algebra. 
We set $\cP := \Add R \subset \sfD(R)$ to be the additive closure of $R$ inside $\sfD(R)$.
Namely, $\cP$ is the full subcategory consisting of $M \in \sfD(R)$ 
which is a direct summand of some coproduct of $R$. 
In sppj resolution, $\cP$ plays the role of projective modules in the usual projective resolution. 

A sppj resolution $P_{\bullet} $ of $M \in \sfD^{< \infty}(R)$ 
is a sequence of exact triangles $\{ \cE_{i}\}_{i \geq 0}$ 
\[
\cE_{i} : M_{i +1} \xrightarrow{g_{i +1}} P_{i} \xrightarrow{f_{i}} M_{i} \to 
\]
such that $f_{i}$ is a sppj morphism (Definition \ref{sppj morphism definition}), 
where we set $M_{0} := M$.  
We often exhibit a sppj resolution $P_{\bullet}$ as below by splicing $\{\cE_{i}\}_{i \geq 0}$ 
\[
P_{\bullet}:  \cdots \to P_{i}  \xrightarrow{ \delta_{i}} P_{i -1} \xrightarrow{ \delta_{i -1}} \cdots \to P_{1} \xrightarrow{\delta_{1}} P_{0} \to M
\]
where we set $\delta_{i} := g_{i} f_{i}$. 
It is analogous to that in the case where $R$ is an ordinary algebra,  
a projective resolution $P_{\bullet}$ of an $R$-module  $M$ is constructed by 
splicing exact sequences 
\[
0 \to M_{i+1} \to P_{i} \to M_{i} \to 0 
\]
with $P_{i}$ projective.

We state the main result which gives equivalent conditions of $\pd M =d$. 
For a DG-$R$-module $M$, we set $\sup M := \sup\{ n \in \ZZ \mid \tuH^{n}(M) \neq 0\}$. 
We denote by $\sfD^{< \infty}(R)$ the derived category of DG-$R$-modules $M$ bounded from above, i.e., $\sup M < \infty$. 
We note that if $M \in \sfD(R)$ has finite projective dimension, then it belongs to $\sfD^{< \infty}(R)$ (Lemma \ref{201711191911}). 

\begin{theorem}[{Theorem \ref{sppj resolution theorem}}]
Let $M \in \sfD^{< \infty}(R)$ and $d \in \NN$ a natural number. 
Then  
the following conditions are equivalent 

\begin{enumerate}[(1)]
\item 
$\pd  M  = d$.

\item 
For any sppj resolution $P_{\bullet}$, 
there exists a natural number $e \in \NN$ 
which satisfying the following properties 

\begin{enumerate}[(a)]
\item $M_{e}$ belongs to $\cP[-\sup M_{e}]$.  

\item 
$d = e+ \sup P_{0} -\sup M_{e}$. 

\item 
The structure morphism $g_{e}: M_{e} \to P_{e-1} $ is not a split-monomorphism. 
\end{enumerate} 

\item 
$M$ has sppj resolution $P_{\bullet}$ of length $e$ 
which satisfies the following properties. 
\[
 P_{e} \xrightarrow{\delta_{e}} P_{e-1} \xrightarrow{\delta_{e-1}} \cdots 
 P_{1} \xrightarrow{\delta_{1}} P_{0} \xrightarrow{f_{0}}  M  
\]

\begin{enumerate}[(a)]
\item 
$d = e+ \sup P_{0} -\sup P_{e}$. 

\item 
The $e$-th differential $\delta_{e}$ is not a split-monomorphism. 
\end{enumerate} 

\item 
 The functor $F= \RHom(M, -)$ sends the standard heart $\Mod \tuH^{0}(R)$ 
to $\sfD^{[-\sup M , d -\sup M]}(R)$ 
and there exists $N \in \Mod H^{0}$ such that $\tuH^{d  -\sup M}(F(N)) \neq 0$.

\item 
$d$ is the smallest number 
which satisfies  
\[
M \in \cP[-\sup M] \ast \cP[-\sup M +1] \ast \cdots \ast \cP[-\sup M  +d]. 
\] 

\end{enumerate}
\end{theorem}

The condition (4) tells that the projective dimension of $M$ can be measured  by only looking the standard heart $\Mod \tuH^{0}(R)$ 
of the derived category $\sfD(R)$. 
The condition (5) says that the projective dimension $\pd M$ is the smallest number of extensions 
by which we obtain $M$ from the ``projective objects" $\cP$ (see Definition \ref{cP definition}).

A similar result for the injective dimension and the flat dimensions are given in Theorem \ref{ifij resolution theorem}
and Theorem \ref{spft resolution theorem}. 

In the final part of this paper,  
we introduce the global dimension $\gldim R$ of a connective DG-algebra $R$. 
For an ordinary ring $R$, a key result to define the global dimension $\gldim R$ is 
that the supremum  of the projective dimensions $\pd M$ of all $R$-modules $M$ and 
that of the injective dimensions $\injdim M$ coincide. 
We provide a similar result for a connective DG-algebra $R$. 
It is well-known that the ordinary global dimensions is not preserved by derived equivalence, 
but their finiteness is preserved. We prove the DG-version of this result.

In the subsequent  paper \cite{CDGA}  we study connective commutative DG-algebras (CDGA), 
more precisely, piecewise Noetherian CDGA, which is a DG-counter part of commutative Noetherian algebra. 
First  we develop basic notion (e.g. depth, localization) and 
establish their properties (e.g. Auslander-
Buchsbaum formula). 
Then, we study (minimal) inf-injective (ifij) resolutions  introduced  in this paper. 
We observe that 
a DG-counter  part $E_{R}(R/\ulfrkp')$  of the class of  indecomposable injective modules 
are  parametrized by  prime ideals $\frkp \in \Spec H^{0}$ of the $0$-th cohomology algebra $H^{0}: =\tuH^{0}(R)$. 
This fact is compatible with the view point of derived algebraic geometry that 
the base affine scheme of the derived affine scheme $\Spec R$ associated to 
a CDGA $R$ is the affine scheme $\Spec H^{0}$. 
We introduce the Bass number $\mu_{R}^{i}(\frkp, M)$  and shows that it describes a minimal ifij resolution of $M$ as in ordinary commutative ring theory. 
One of the main result  is a structure theorem of a minimal ifij resolution of a dualizing complex $D$. 

\begin{theorem}[{\cite{CDGA}}]
Let $R$ be a connective piecewise Noetherian CDGA.  
Assume that $R$ has a dualizing complex  $D \in \sfD(R)$ with a minimal ifij resolution $I_{\bullet}$ of length $e$.  
Then  $H^{0}:= \tuH^{0}(R)$ is catenary and $\dim H^{0} < \infty$. 
If moreover we assume that $\tuH^{0}(R)$ is local, 
then  the following statements hold. 
\begin{enumerate}[(1)]

\item  $\inf I_{-i} = \inf D$ for $i = 0, \cdots, e$.  

\item $e = \injdim D =\depth D = \dim H^{0}$.

\item 
\[
I_{-i} = \bigoplus_{\frkp } E_{R}(R/\ulfrkp')[-\inf D] 
\]
where $\frkp$ run all prime ideals such that $i = \dim H^{0} - \dim H^{0}/\frkp$. 
 \end{enumerate}
 \end{theorem}

This statement is completely analogues to the structure theorem of a minimal injective resolution of 
a dualizing complex, 
which is  one of the fundamental result in  classical commutative ring theory  
 proved in \cite{Residues and Duality}, summarized in \cite[Theorem 4.2]{Foxby:Complexes}. 
In \cite{CDGA} we  see that 
not only this theorem but also other results about minimal ifij resolutions 
are analogous  to  classical results about minimal injective resolution.
This fact supports that the ifij resolution is a proper generalization of the injective resolution.  
We can expect that  it become an indispensable tool for studying DG-modules 
like an ordinary injective resolutions for studying modules.

The paper is organized as follows. 
In Section \ref{Resolutions of DG-modules}, we introduce and study a sppj resolution.  
Section \ref{ifij resolution} deals with an inf-injective(ifij) resolution.   
Since the basic properties are proved in  the same way with that of the similar statement of sppj resolutions, most of all proofs are omitted. 
However we need to study the class $\cI$ which plays the role of injective modules for a ordinary injective resolution. 
This class of DG-modules was already studied by Shaul in \cite{Shaul:Injective}, in which he denoted $\cI$ by $\Inj R$. 
But we take a different approach to the class $\cI$.  
Section \ref{spft resolution} deals with a sup-flat (spft) resolution. 
A description of the class $\cF$ of flat dimension $0$ is still a conjecture. 
In Section \ref{The global dimension} we introduce the global dimension $\gldim R$ of a connective DG-algebra $R$ 
and prove that finiteness is preserved by derived equivalence.

\subsection{Notation and convention}

The basic setup and notations are the followings.

Throughout the paper, we fix a base commutative ring $\kk$ and  (DG, graded) algebra is  (DG, graded) algebra over $\kk$.
We denote by $R =(R, \partial) $   a connective cohomological  DG-algebra. 
Recall that  ``connective" means that $\tuH^{> 0}(R) = 0$.  
We note that every connective DG-algebra $R$ is quasi-isomorphic to a DG-algebra $S$ such that $S^{>0} =0$. 
Since quasi-isomorphic DG-algebras have equivalent derived categories, 
it is harmless to assume that $R^{>0} = 0$ for our purpose. 

For simplicity we denote by $H := \tuH(R)$ the cohomology algebra of $R$, 
by $H^{0} := \tuH^{0}(R)$ the $0$-th cohomology algebra of $R$. 
We denote by $\GrMod H$ the category of graded $H$-modules, 
by $\Mod H^{0}$ the category of $H^{0}$-modules. 

We denotes by  $\sfC(R)$  the category of DG-$R$-modules and cochain morphisms, 
by $\sfK(R)$  the homotopy category of DG-$R$-modules 
and by  $\sfD(R)$ the derived category of DG $R$-modules. 
The symbol $\Hom$ denotes the $Hom$-space of $\sfD(R)$. 

Let $n \in \{ -\infty \} \cup \ZZ \cup \{ \infty\}$. 
The symbols $\sfD^{<n}(R)$, $\sfD^{>n}(R)$ denote the full subcategories of $\sfD(R)$ consisting of $M$ such that 
$\tuH^{\geq n}(M) = 0, \ \tuH^{\leq n}(M) = 0$ respectively. 
We set 
$\sfD^{[a,b]}(R) = \sfD^{\geq a}(R) \cap \sfD^{\leq b}(R)$ 
for $a,b \in \{ - \infty \} \cup \ZZ \cup \{ \infty\}$ such that $a \leq b$. 
We set  $\sfD^{\mrb}(R) := \sfD^{< \infty}(R) \cap \sfD^{> -\infty}(R)$. 

Since $R$ is connective, the pair $(\sfD^{\leq 0}(R), \sfD^{\geq 0}(R))$ is a $t$-structure in $\sfD(R)$, 
which is called the \textit{standard} $t$-structure. 
The truncation functors  are denote by  $\sigma^{<n}, \sigma^{>n}$.  
We identify the heart $\sfH = \sfD^{\leq 0}(R) \cap \sfD^{\geq 0}(R)$ 
 of the standard  $t$-structure with  $\Mod H^{0}$ via the functor $\Hom(H^{0}, -)$, 
 which fits into the following commutative diagram 
 \[
 \begin{xymatrix}{ 
  && \sfH \ar[drr]^{\mathsf{can}} \ar[dd]_{\cong}^{\Hom(H^{0},-)} && \\ 
 \sfD(R) \ar[urr]^{\tuH^{0}} \ar[drr]_{\Hom(R, -)}  &&&& \sfD(R) \\ 
 && \Mod H^{0},  \ar[urr]_{f_{*}} &&
 }\end{xymatrix}
 \]
 where $\mathsf{can}$ is the canonical inclusion functor 
 and $f_{*}$ is the restriction functor along a canonical projection $f: R \to H^{0}$.

For a DG-$R$-module $M\neq 0$, we set
$\inf M := \inf \{ n \in  \ZZ \mid \tuH^{n}(M) \neq 0\}$,  
$\sup M := \sup \{ n \in \ZZ \mid \tuH^{n}(M) \neq 0\}$, 
$\champ M := \sup M - \inf M$. 
In the case $\inf M > - \infty$, we use the abbreviation $\tuH^{\inf}(M) := \tuH^{\inf M }(M)$. 
Similarly in the case $\sup M < \infty $, we use the abbreviation $\tuH^{\sup}(M) := \tuH^{\sup M}(M)$. 
We formally set $\inf 0:= \infty$ and $\sup 0 := - \infty$. 

In the case where we need to indicate the DG-algebra $R$, we denote $\sup_{R} M, \inf_{R} M$ and $\champ_{R} M$.

\vspace{10pt}
\noindent
\textbf{Acknowledgment}


The author would like to thank L. Shaul for his comments on the first draft of this paper 
which helped to improve many points.  
He also thanks M. Ono for drawing  his attention to  the papers by Shaul, 
and for comments on earlier versions of this paper.   
He expresses his thanks  to I. Iwanari and O. Iyama for comments on the first version of this paper, to S. Yasuda for suggesting Theorem \ref{global dimension theorem}. 
to I. Kikumasa for giving a question which leads to Proposition \ref{global dimension 0}.  
He is grateful to  A. Yekutieli for pointing out that the definition of projective concentration 
in earlier versions was opposite to the his original definition.  
The author  was partially  supported by JSPS KAKENHI Grant Number 26610009.

\section{Sup-projective resolutions of DG-modules}\label{Resolutions of DG-modules}


\subsection{Projective dimension of $M \in \sfD(R)$ after Yekutieli}

We recall the definition of the projective dimension of $M \in \sfD(R)$ introduced by Yekutieli. 

\begin{definition}[{\cite[Definition 2.4]{Yekutieli}}]
Let  $a \leq b \in \{ -\infty \} \cup \ZZ \cup \{\infty\}$. 

\begin{enumerate}[(1)] 
\item 
An object $M \in \sfD(R)$ 
is said 
to have \textit{projective concentration} $[a,b]$ 
if  the functor  $F = \RHom_{R}(M, -)$ 
sends $\sfD^{[m,n]}(R)$ to 
$\sfD^{[m -b,n -a ]}(\kk)$ 
for any $m \leq n \in \{-\infty\} \cup \ZZ \cup \{\infty\}$.
\[
F(\sfD^{[m,n]}(R)) \subset \sfD^{[m -b,n -a]}(\kk).
\]

\item 
An object $M \in \sfD(R)$ 
is said 
to have \textit{strict projective concentration} $[a,b]$
if it has projective concentration $[a,b]$ 
and does't have projective concentration $[c,d]$ 
such that $[c,d] \subsetneq [a,b]$. 

\item 
An object $M \in \sfD(R)$ 
is said 
to have projective dimension $d \in \NN$ 
if it has strict projective concentration $[a,b]$ for $a,b\in \ZZ$.
such that $d= b-a$. 

In the case where, $M$ does't have a finite interval as  projective concentration, 
it is said to have infinite projective dimension. 

We denote the projective dimension by $\pd M$.
\end{enumerate}
\end{definition}

\begin{remark}\label{AF projective dimension}
Let $R$ be an ordinary algebra. 
Avramov-Foxby \cite{Avramov-Foxby:Homological dimensions} introduced 
another   projective dimension for a complex $M \in \sfC(R)$. 
If we denote by $\textup{AF}\pd M$ the projective dimension of Avramov-Foxby, 
then it is easy to see that $\textup{AF}\pd M = \pd M -\sup M$. 
\end{remark}

The following lemma is proved later after the proof of Theorem \ref{Shaul theorem}.
\begin{lemma}\label{201711191911}
If $M \in \sfD(R)$ has finite projective dimension, 
then it belongs to $\sfD^{<\infty}(R)$. 
\end{lemma}

The following lemma is deduced from the property of the standard $t$-structure of $\sfD(R)$.
\begin{lemma}\label{201712182258}
Let $M \in \sfD^{< \infty}(R)$ 
and $F:= \RHom(M, -)$. 
Then for all $m \leq n$, 
\[
F(\sfD^{[m,n]}(R)) \subset \sfD^{[m -\sup M, \infty]} (R)
\] 
and there is $N \in \sfD^{[m,n]}(R)$ such that $\tuH^{m- \sup M}(F(N) ) \neq 0$. 
\end{lemma}

\begin{proof}
Let $ \ell < m -\sup M$. Then 
$\tuH^{\ell}(F(N)) = \Hom(M,N[\ell]) = 0$ 
for $N \in \sfD^{[m,n]}(R)$. 
This proves the first statement. 

The standard $t$-structure induces a nonzero morphism $M \to \tuH^{\sup}(M)[-\sup M]$. 
Thus,  $N = \tuH^{\sup}(M)[-m]$ has the desired property. 
\end{proof}

We deduce the following useful corollaries. 

\begin{corollary}\label{201712182258II}
An object  $M \in \sfD^{<\infty}(R)$ has projective dimension $d$ 
if and only if
it has strict projective concentration $[\sup M -d, \sup M]$. 
\end{corollary}

\begin{corollary}\label{201804030002}
Let $L \to M  \to N \to $  be an exact triangle in $\sfD(R)$. 
Then, we have 
\[
\pd M - \sup M \leq \sup \{ \pd L - \sup L, \pd N - \sup N\}.
\]
\end{corollary}

\subsection{The class $\cP$ and sup-projective (sppj) resolution}

The class $\cP$ plays the role of projective modules in the usual projective resolutions  for  sup-projective resolutions.

\begin{definition}\label{cP definition}
We denote by  $\cP \subset \sfD(R)$  the full subcategory of 
direct summands of a direct sums of $R$. 
In other words, $\cP = \Add R$.  
\end{definition}

The basic properties of $\cP$ are summarized in the lemma below. 
By $\Proj H^{0}$ we denote the full subcategory of projective $H^{0}$-modules. 

\begin{lemma}\label{basic of cP lemma} 
\begin{enumerate}[(1)]
\item
For $N \in \sfD(R)$ and $P \in \cP$, the map below associated to the $0$-th cohomology functor $\tuH^{0}$ is an isomorphism.
\[
\Hom(P, N) \xrightarrow{\cong} \Hom(\tuH^{0}(P), \tuH^{0}(N)).  
\]

\item
For $N \in \Mod H^{0}$, we have 
\[
\Hom(P, N[n]) =
\begin{cases}
\Hom_{\Mod H^{0}}(\tuH^{0}(P), N) & n = 0, \\ 
0  & n \neq 0.
\end{cases}
\]

\item 
The functor $\tuH^{0}$ induces an equivalence $\cP \cong \Proj H^{0}$. 
\end{enumerate}
\end{lemma}

\begin{remark}
Lurie \cite[Section 7.2]{Lurie:HA} studied the class $\cP$ for $\mathbb{E}_{1}$-algebras 
and obtained the same result with (3) of above lemma in \cite[Corollary 7.2.2.19]{Lurie:HA}. 
\end{remark}

\begin{proof}
(1) is a consequence of the isomorphism $\Hom(R, N) \cong \tuH^{0}(N)$ for $N \in \sfD(R)$. 

(2) is an immediate consequence of (1) 

(3) By (1), the functor $\tuH^{0}: \cP \to \Proj H^{0}$ is fully faithful. 
We prove it is essentially surjective. 
Let $Q \in \Proj H^{0}$. 
We need to show that there exists $\underline{Q} \in \cP$ such that 
$\tuH(\underline{Q} ) \cong Q$. 
The case where $Q$ is a free $H^{0}$-module is clear. 
We deal with the general cases. 
Then, 
there exists a free $H^{0}$-module $F$ and an idempotent element $e \in \End_{H^{0}}(F)$ whose kernel is $Q$. 
Let $\underline{F} \in \cP$ be such that $\tuH(\underline{F}) \cong F$. 
By (1), there exists an idempotent element $\underline{e} \in \End \underline{F}$ 
which is sent to $e$ by the map associated to $\tuH^{0}$. 
It is easy to check that the direct summand of $\underline{F}$ corresponds to $\underline{e}$ 
has the desired property. 
\end{proof}

\begin{remark}\label{cP remark}
Almost all the result of the rest of this section are deduced from the properties of $\cP$ given in Lemma \ref{basic of cP lemma}.  
We can state and prove these results in an abstract setting of a triangulated category with $t$-structure whose heart has enough projectives. 
\end{remark}

By (1), it is clear that any object $P \in \cP$ has has projective dimension $0$. 
Using this fact and standard argument of triangulated category, 
we obtain the following corollary. 

Recall that for full subcategories $\cX, \cY \subset \sfD(R)$, we define a full subcategory 
$\cX \ast \cY$ to be the full subcategory consisting $Z \in \sfD(R)$ 
which fits into an exact triangle 
$X \to Z \to Y \to $ with $X \in \cX,  Y \in \cY$. 
Note that this operation is associative, i.e., $(\cX \ast \cY) \ast \cZ = \cX \ast(\cY \ast \cZ)$. 
We remark that if $Z \in \sfD(R)$ satisfies 
$\Hom(Z, \cX) = 0, \ \Hom(Z, \cY) = 0$, 
then 
$\Hom(Z, \cX \ast \cY ) = 0$.

\begin{corollary}\label{20171219208}
Let $a ,b \in \ZZ$ such that $a \leq b$. 
Then for any object $M \in \cP[a] \ast \cP[a + 1] \ast \cdots \ast \cP[b]$, 
we have $\pd M \leq b-a$.  
\end{corollary}

When we construct a usual projective resolution for an ordinary module, 
we use surjective homomorphisms from a projective modules. 
Next we introduce the notion which plays that role for our sup-projective resolution.

\begin{definition}[sppj morphisms]\label{sppj morphism definition} 
Let $M \in \sfD^{< \infty}(R), M \neq 0$. 

\begin{enumerate}
\item 
A sppj  morphism $f: P \to M $ is a morphism in $\sfD(R)$ 
such that $P \in \cP[- \sup M]$ 
and the morphism $\tuH^{\sup M}(f)$ is surjecitve. 

\item 
A sppj morphism $f: P \to M $ is called minimal  if 
the morphism $\tuH^{\sup M}(f) $ is a projective cover. 
\end{enumerate}
\end{definition}

By Lemma \ref{basic of cP lemma},  for any $M \in \sfD^{< \infty}(R)$, there exists a sppj morphism $f: P \to M$. 
The following two lemmas give a motivation to introduce  sppj resolutions.

\begin{lemma}\label{20171219228}
Let $M \in \sfD^{< \infty}(R)$ and $f: P \to M$ a sppj morphism 
and $N := \cone (f)[-1]$ the cocone of $f$. 
Assume that $1\leq \pd M $. 
Then,   $\pd N = \pd M -1 -\sup M +\sup N$. 
\end{lemma}

\begin{proof}
Set $d = \pd M, F= \RHom(M,-)$ and let $g: N \to P$  be the canonical morphism. 
We may assume that $\sup M = 0$ by shifting the degree. 

We claim that $N$ has projective concentration $[1-d, 0]$. 
Let  $m \leq n \in \{-\infty \} \cup \ZZ \cup \{ \infty \}$ and $L \in \sfD^{[m,n]}(R)$. 
We need to prove  $\Hom(N,L[i]) = 0$ for $i \in \ZZ \setminus [m, n+d -1]$.

 By the assumption, 
 we have $\Hom(M,L[i]) = 0$ for $i \in \ZZ \setminus [m, n+d]$. 
 We also have $\Hom(P,L[i]) = 0$ for $i \in \ZZ\setminus [m,n]$.  
 We consider the exact sequence 
 \begin{equation}\label{201712182246}
 \Hom(P, L[i]) \to \Hom(N,L[i]) \to \Hom(M, L[i+1]) \to \Hom(P,L[i+1]).  
 \end{equation}
 Using this, we can  check that  $\Hom(N,L[i]) = 0$ for $i \in \ZZ\setminus [m-1, n +d -1]$. 
It remains to prove the case $i = m-1$. 
In that case the exact sequence (\ref{201712182246}) become 
 \begin{equation}\label{201712182246II}
 0 \to \Hom(N,L[m-1]) \to \Hom(M, L[m]) \xrightarrow{f_{*}} \Hom(P,L[m]).  
 \end{equation}
Under the isomorphisms,  
\[
\Hom(M,L[m]) \cong \Hom(\tuH^{0}(M), \tuH^{m}(L)), \ \ \Hom(P, L[m]) \cong \Hom(\tuH^{0}(P), \tuH^{m}(L))
\]
the map $f_{*}$ corresponds to the map $\Hom(\tuH^{0}(f), \tuH^{m}(L))$, which is injective. 
Thus, we conclude that $\Hom(N, L[m-1]) = 0$. 
This completes the proof of the claim.

By Lemma \ref{201712182258II} and the claim, $N$ has projective concentration $[1-d, \sup N]$. 
Let $[a,b]$ be a strict projective concentration of $N$. 
Then, by Lemma \ref{201712182258II}, $ 1 -d \leq a = \sup N - \pd N, b = \sup N$. 
Therefore to prove the desired formula $\pd N = d-1 + \sup N$. 
It is enough to show that 
there exists $m\leq n \in \{-\infty\} \cup \ZZ \cup \{ \infty\}$ 
and $L \in \sfD^{[m,n]}(R)$ such that $\Hom(N, L[n+d -1]) \neq 0$. 

Since $d = \pd M$, there exists $m\leq n \in \{-\infty\} \cup \ZZ \cup \{ \infty\}$ 
and $L \in \sfD^{[m,n]}(R)$ such that $\Hom(M, L[n+d]) \neq 0$. 
Since we assume that $ d \leq 1$, we have $\Hom(P, L[n +d]) = 0$. 
 Therefore from the $i = n + d -1$ case of the  exact sequence (\ref{201712182246}),  
we deduce that there exists a surjection $\Hom(N,L[n+ d-1]) \twoheadrightarrow \Hom(M,L[n +d]) \neq 0$.  
 Thus, we conclude that $\Hom(N, L[n +d -1]) \neq 0$ as desired. 
\end{proof}

\begin{lemma}\label{20171219205II} 
Let $M \in \sfD(R) \setminus \{ 0\}$. 
Then $\pd M = 0$ if and only if $M \in \cP[-\sup M] $. 
\end{lemma}

\begin{proof}
``if" part is clear. 
We prove ``only if" part. 
Let $M \in \sfD(R)$ be such that $\pd M = 0$. 
By \ref{201712182258}, $\sup M < \infty$. 
Shifting the degree, we may assume $\sup M = 0$. 
We take an exact triangle below with $f$ sppj  
\begin{equation}\label{20171219135}
N \xrightarrow{g} P \xrightarrow{f} M \to.  
\end{equation}
Since $\cP$ is closed under taking direct summand, 
it is enough to show that $g$ is a split monomorphism.

Observe that $N \in \sfD^{\leq 0}(R)$. 
The above exact triangle induces the exact sequence 
\[
 0\to \Hom(M, \tuH^{0}(N)) \to \Hom(P, \tuH^{0}(N))\xrightarrow{g_{*}}\Hom(N, \tuH^{0}(N)) \to 0
\]
where we use $\Hom(M, \tuH^{0}(N)[1]) = 0$. 
Under the isomorphisms  
\[
\Hom(P , \tuH^{0}(N)) \cong \Hom(\tuH^{0}(P), \tuH^{0}(N)), \  
\Hom(N , \tuH^{0}(N)) \cong \Hom(\tuH^{0}(N), \tuH^{0}(N)) 
\]
the morphism $g_{*}$ corresponds  the morphism $\Hom(\tuH^{0}(g), \tuH^{0}(N))$. 
Let $h: N \to \tuH^{0}(N)$ be a canonical morphism. 
Then, there exists a morphism $k :P \to \tuH^{0}(N)$ such that $kg = h$. 
Note that $h \in \Hom(N, \tuH^{0}(N))$  corresponds to $\id_{\tuH^{0}(N)} \in \Hom(\tuH^{0}(N), \tuH^{0}(N))$ 
under the above isomorphism 
and hence  $\tuH^{0}(k) \tuH^{0}(g) = \id_{\tuH^{0}(N)}$. 
Since $h$ induces an isomorphism $\Hom(P, N) \xrightarrow{ \cong } \Hom(P, \tuH^{0}(N))$, 
there exists a morphism $\ell: P \to N$ such that $h\ell = k$. 
\[
\begin{xymatrix}
{
N \ar[rr]_{g} \ar[d]_{h} &&  P \ar@/_1pc/[ll]_{\ell} \ar[dll]^{k} \\
\tuH^{0}(N),  
}\end{xymatrix}
\begin{xymatrix}{
Q \ar[drr]^{i} \ar[d]^{\ell i}  && \\ 
N \ar@/^1pc/[u]^{jg} \ar[rr]_{g} && P.  
 }\end{xymatrix}
\] 
Check that $\tuH^{0}(\ell) \tuH^{0}(g) = \id_{\tuH^{0}(N)}$. 
Therefore if we set $ e = g \ell$, 
then $\tuH^{0}(e)$ is an idempotent element of $\End(\tuH^{0}(P))$. 
Since the $0$-th cohomology group functor $\tuH^{0}$ induces an isomorphism  
$\End(P) \cong  \End(\tuH^{0}(P))$, we conclude that $e$ is an idempotent element.  
Let $Q$ be the corresponding direct summand of $P$. 
More precisely, $Q$ is an object of $\sfD(R)$ equipped with morphisms $i :Q \to P$ and $j : P \to Q$ 
such that $ji = \id_{Q}$ and $ij = e$. 
Then, we have $g\ell i = ei = iji= i$ and $jg\ell i = \id_{Q}$. 
This shows that there exists an isomorphism $N \cong Q \oplus N'$ for some $N' \in \sfD(R)$ 
under which the morphism $g$ corresponds to $\begin{pmatrix} i \\ g' \end{pmatrix}$ for some $g' : N' \to P$.
\[
g: N \cong Q \oplus N' \xrightarrow{  \ \ \ \begin{pmatrix} i \\ g' \end{pmatrix} \ \ \ } P. 
\]
If we show $N' = 0$, then $Q \cong N$ and finish the proof. 

Assume on the contrary that $N' \neq 0$. 
First observe that 
$\tuH^{0}(Q) = \tuH^{0}(N)$  by construction. 
Therefore, $s:= \sup N' < 0$.   
We have 
$\Hom(Q \oplus N', \tuH^{s}(N)[-s]) \neq 0$. 
On the other hand, 
applying $\Hom(-, \tuH^{s}(N)[-s])$ to the exact triangle (\ref{20171219135}), 
we obtain an exact sequence 
\[
\Hom(P, \tuH^{s}(N)[-s]) \to \Hom(N, \tuH^{s}(N)[-s]) \to \Hom(M, \tuH^{s}(N)[-s+1]).    
\] 
Since $s < 0$, both sides of this sequence is zero. 
Hence we conclude that $\Hom(Q \oplus N', \tuH^{s}(N)[-s]) =0$.
A contradiction. 
\end{proof}

In the proof we showed that 

\begin{corollary}\label{20171219205}
Let $M \in \sfD(R)$ be such that $\pd M = 0$ and $f: P \to M$ a sppj morphism. 
Then, the cocone $N = \cone f [-1] $ has $\pd N = 0$. 
\end{corollary}

It is worth noting the following corollary. 

\begin{corollary}\label{201901042037}
The full subcategory $\cP \subset \sfD(R)$ is 
consisting of objects $P$ such that either $P = 0$ or  
$\pd P = 0$ and $\sup P = 0$.
\end{corollary}

Now from Lemma \ref{20171219228} and Lemma \ref{20171219205II} 
it is clear how to define a resolution of a DG-module which computes its projective dimension.

\begin{definition}[sppj resolutions]
\begin{enumerate}
\item 
A sppj resolution $P_{\bullet}$ of $M$ is a sequence of exact triangles  for $i \geq  0$ 
with $M_{0} := M$ 
\[
M_{i+1} \xrightarrow{g_{i+1}} P_{  i } \xrightarrow{f_{i}} M_{i}
\]
such that $f_{i}$ is sppj.

The following inequality folds
\[
\sup M_{i+1} = \sup P_{i+1} \leq \sup P_{i} =\sup M_{i}. 
\]

For a sppj resolution $P_{\bullet}$ with the above notations, 
we set $\delta_{i} := g_{i-1} \circ f_{i}$. 
\[
\delta_{i} :P_{ i} \to P_{i-1}.
\]
Moreover we write 
\[
 \cdots \to P_{i} \xrightarrow{\delta_{i}} P_{ i -1} \to \cdots \to P_{1} \xrightarrow{\delta_{1}} P_{0} \to M. 
\]

\item 
A sppj resolution $P_{\bullet}$ is said to have length $e$ if 
$P_{i } = 0$ for $i> e$ and 
$P_{e} \neq 0$.

\item 
A sppj resolution $P_{\bullet}$ is called minimal if $f_{i}$ is minimal 
for $i \geq 0$.

\end{enumerate}
\end{definition}

Using sppj resolution, we can compute $\Hom(M,N[n])$ for $N \in \mod H^{0}$.  

\begin{lemma}\label{sppj resolution lemma}
Let $N \in \mod H^{0}, M \in \sfD^{< \infty}(R)$ and $P_{\bullet}$  a sppj resolution of $M$. 
We denote the complexes below by $X_{i}, X'_{i}$. 
\[
\begin{split}
X_{i} : \Hom(P_{i-1}, N[-\sup P_{i}]) \to \Hom(P_{i }, N[-\sup P_{i}] ) \to \Hom(P_{i+1}, N[ -\sup P_{i}]) \\ 
X'_{i} :\Hom(\tuH^{\sup}(P_{i-1}), N) \to \Hom(\tuH^{\sup}(P_{ i }), N) \to \Hom(\tuH^{\sup}(P_{i +1}), N) 
\end{split}
\] 
Then, 
\[
\Hom(M,N[n]) =
\begin{cases} 
0 & n \neq i -\sup P_{i} \textup{ for any } i \geq 0 \\
\tuH(X_{i}) & n = i - \sup P_{i} \textup{ for some } i \geq 0 
\end{cases}
\]

Moreover, in the case $n = i -  \sup P_{i}$,  we have  
\[
\begin{split}
&\tuH(X_{i}) \\
&=
\begin{cases}
\Hom(\tuH^{\sup}(P_{i}), N) & \sup P_{i -1} \neq \sup P_{i} \neq \sup P_{i+1}, \\
\Ker[ \Hom(\tuH^{\sup}(P_{i }), N ) \to \Hom(\tuH^{\sup}(P_{i+1}) , N) ],  &
 \sup P_{i-1} \neq \sup P_{i} = \sup P_{i+1}, \\ 
 \Coker[\Hom(\tuH^{\sup}(P_{i -1}), N) \to \Hom( \tuH^{\sup}(P_{i}), N ) ], & 
 \sup P_{i -1} = \sup P_{i} \neq \sup P_{i+1}, \\ 
\tuH(X'_{i}), & 
 \sup P_{i -1} = \sup P_{i} = \sup P_{i +1}. 
\end{cases}  
\end{split}
\]
\end{lemma}

We note that for $i \neq j$, 
we have 
$i - \sup P_{i} \neq j -\sup P_{j}, 
i - \sup P_{i} + 1 \neq j -\sup P_{j} +1$. 
We also not that 
a pair  $i,j \geq 0$ satisfies $i - \sup P_{i} + 1= j - \sup P_{j}$ 
if and only if $ j= i+ 1$ and $ \sup P_{i} = \sup P_{j}$.

\begin{proof}
For simplicity we set $t_{i} := - \sup P_{i}$. 
We have 
\[
\Hom(P_{i} , N [n]) \cong
 \begin{cases}
\Hom_{\Mod H^{0}}(\tuH^{\sup}(P_{i}), N)  & n = t_{i}, \\
0 & n \neq t_{i}. 
\end{cases}
\]

From the induced exact sequence 
\[
\Hom(P_{i} , N[n-1]) \to \Hom(M_{i+1}, N[n-1]) \to \Hom(M_{i}, N[n]) \to \Hom(P_{i}, N[n]), 
\]
we deduce the following isomorphism and exact sequences. 

For $n \neq t_{i}, t_{i} + 1$, we have an isomorphism
\begin{equation}\label{201711121859}
\Hom(M_{i+1}, N[n-1]) \xrightarrow{\cong } \Hom(M_{i}, N[n]). 
\end{equation}

We have the exact sequence 
\begin{equation}\label{201711121858}
\begin{split}
0 = \Hom(M_{i+1}, N[t_{i} -1]) \to \Hom(M_{i}, N[ t_{i} ]) \to \Hom(P_{i}, N[t_{i}]) \\
 \to \Hom(M_{i+1}, N[t_{i} ]) \to \Hom(M_{i}, N[ t_{i} + 1]) \to 0,
\end{split}
\end{equation}
The first term is zero since $\sup M_{i+1} \leq \sup P_{i}$. 
We note that in a similar way we see that 
the induced map 
$ \Hom(M_{i+1}, N[t_{i}]) \to \Hom(P_{i+1}, N[ t_{i}] ) $ is injective.

We have the exact sequence 
\begin{equation}\label{201711162106}
\begin{split}
\Hom(P_{i-1}, N[ t_{i}]) \to \Hom(M_{i}, N[ t_{i} ]) \to \Hom(M_{i-1}, N[t_{i}+1]) \\
 \to \Hom(P_{i -1}, N[t_{i} + 1]) = 0 
\end{split}
\end{equation}
The last term is zero, since $t_{i} + 1 > t_{i -1}$.

Combining above observations, we obtain the following diagram,  
whose row and columns are exact. 
\[
\begin{xymatrix}{ 
& \Hom(P_{i-1}, N[t_{i}] ) \ar[d] \ar[dr] & &  0 \ar[d] \\
 0 \ar[r] & 
 \Hom(M_{i}, N[t_{i}]) \ar[r] \ar[d] & 
 \Hom(P_{i}, N[t_{i}]) \ar[r] \ar[dr] & \Hom(M_{i+1}, N[t_{i}] ) \ar[d] \\ 
 &\Hom( M_{i-1}, N[t_{i} + 1]) \ar[d] && \Hom(P_{i+1}, N[t_{i}] ) \\ 
& 0& & }\end{xymatrix} 
 \]
Observe that the slant line is the complex $X_{i}$. 
Thus we conclude that $\tuH(X_{i} ) \cong \Hom(M_{i -1}, N[t_{i} +1])$. 


We deal with the case $n = i + t_{i}$ for some $i \geq 0$. 
Then, $n \neq j+t_{j} +1, j + t_{j}$ for $j = 0, \dots,  i-2$. 
Therefore  
we have the following isomorphisms. 
\begin{equation*}
\begin{split}
\Hom(M,N[n]) \cong \Hom(M_{1}, N [n-1]) &\cong \cdots \cong  \Hom(M_{i-1}, N[n-i +1])  \\
\end{split}
\end{equation*}
Since $n -i +1 = t_{i} + 1$, 
we obtain the desired isomorphism 
\[
\Hom(M,N[n]) \cong \Hom(M_{i-1}, N[t_{i} + 1] ) \cong \tuH(X_{i}). 
\]

Next we deal with the case $n \neq  i +t_{i}$ for any $i \geq 0$. 
This case divided into the following two cases:
 
(I) $ n = i + t_{i} +1$ for some $i \geq 0$. 

(II) $ n \neq i+ t_{i} +1$ for any $ i \geq 0$. 

We deal with the case (I). 
Then, $n \neq j+t_{j} +1, j + t_{j}$ for $j = 0, \dots,  i-1$. 
Therefore  
we have the following isomorphisms. 
\begin{equation*}
\begin{split}
\Hom(M,N[n]) \cong \Hom(M_{1}, N [n-1]) &\cong \cdots \cong  \Hom(M_{i}, N[n-i ])  \\
\end{split}
\end{equation*}
Since $ n-i = t_{i} +1$, 
\[
\Hom(M,N[n]) = \Hom(M_{i}, N[ t_{i} + 1]) 
\]
We only have to show that $\Hom(M_{i}, N[t_{i} + 1]) = 0$. 
We have $t_{i} \neq t_{i +1}$, since otherwise $n= (i+1) + t_{i +1}$. 
Thus $\Hom(P_{i+1}, N[t_{i}]) = 0$. 
As we mentioned before the induced map 
$\Hom(M_{i+1}, N[t_{i}]) \to \Hom(P_{i+1}, N[t_{i}])$  
is injective. Thus, $\Hom(M_{i +1}, N[t_{i}]) = 0$. 
Finally, 
from the exact sequence 
(\ref{201711121858}), 
we obtain a 
surjection $\Hom(M_{i+1}, N[t_{i}]) \to \Hom(M_{i}, N[t_{i} + 1])$. 
Thus, we conclude 
$\Hom(M_{i}, N[t_{i} + 1]) = 0$ as desired.

We deal with the case (II). 
Namely we assume $n \neq i + t_{i}, i +t_{i} + 1$ for any $i \geq 0$. 
This case is divided into the following two cases: 

(II-i) $n < t_{0}$. (II-ii) $ n > t_{0}$. 

The case (II-i). 
Since $M \in \sfD^{\leq -t_{0}}(R) , N[n]  \in \sfD^{ > - t_{0}}(R)$, 
we have $\Hom(M, N[n]) = 0$. 

The case (II-ii). 
Let $ i= n -t_{0} +1$. 
The we have isomorphisms
\[
\begin{split}
\Hom(M,N[n]) \cong \Hom(M_{1}, N [n-1]) &\cong \cdots \cong  \Hom(M_{i}, N[n-i]) \\
\end{split}
\] 
Since $\sup M_{i} = \sup P_{i} = -t_{i} \leq -t_{0}< -(n -i)$, 
using  the standard $t$-structure  as above, 
we deduce that $\Hom(M_{i}, N[n-i]) = 0$. 
\end{proof}

The following theorem provides criteria 
for a natural number to be an upper bound of the projective dimension $\pd M$ 
of a DG-$R$-module $M$.

\begin{theorem}\label{sppj resolution theorem lemma II}
Let $M \in \sfD^{< \infty}(R)$ and $d \in \NN$ a natural number. 
Then,   
the following conditions are equivalent 

\begin{enumerate}[(1)]
\item 
$\pd  M  \leq  d$. 

\item 
For any sppj resolution $P_{\bullet}$, 
there exists a natural number $e \in \NN$ 
such that 
$M_{e} \in \cP[-\sup P_{e}]$ 
and 
$ e+ \sup P_{0} -\sup M_{e}\leq d$. 
In particular, 
we have a sppj resolution of length $e$. 
\[
 M_{e} \to P_{e-1} \to P_{e-2} \to \cdots P_{1} \to P_{0} \to M.
\]

\item 
$M$ has sppj resolution $P_{\bullet}$ of length $e$ 
such that $ e + \sup P_{0} - \sup P_{e} \leq d$.

\item 
The functor $F= \RHom(M, -)$ sends the standard heart $\Mod H^{0}$ 
to $\sfD^{[-\sup M, d  -\sup M  ]}(R)$. 

\item 
$M$ belongs to $\cP[-\sup M ] \ast \cP[ -\sup M +1] \ast \cdots \ast \cP[-\sup M +d]$. 

\end{enumerate}
\end{theorem}

We need  a preparation. 

\begin{lemma}\label{201711171950}
Let $Q_{1}, Q_{2}, Q_{3}$ be objects of $\cP$,  
 $N_{3} \xrightarrow{g_{3}} Q_{2} \xrightarrow{f_{2}} N_{2} \to $ be an exact triangle  in $\sfD(R)$ and  
$f_{3}: Q_{3} \to N_{3}, g_{2} : N_{2} \to Q_{1}$ morphisms in $\sfD(R)$. 
We set $\delta_{3} := g_{3}f_{3}, \delta_{2}: = g_{2}f_{2}$. 
We consider the complex $Q_{3} \xrightarrow{\delta_{3}} Q_{2} \xrightarrow{\delta_{2}} Q_{1}$ inside $\sfD(R)$. 
Assume 
for any $N \in \Mod H^{0}$, the induced complex below is exact.  
\begin{equation}\label{201901041818}
\Hom_{\Mod H^{0}}(\tuH^{0}(Q_{1}), N) \to 
\Hom_{\Mod H^{0}}(\tuH^{0}(Q_{2}), N) \to 
 \Hom_{\Mod H^{0}}(\tuH^{0}(Q_{3}), N)
\end{equation}

Then the complex $Q_{3} \xrightarrow{\delta_{3}} Q_{2} \xrightarrow{\delta_{2}} Q_{1}$ inside $\sfD(R)$ splits as below. 
\begin{equation}\label{201804052136}
\begin{xymatrix}{
Q_{3} \ar[rr]^{\delta_{3}} \ar@{=}[d]^{\wr} &&
Q_{2} \ar[rr]^{\delta_{2}} \ar@{=}[d]^{\wr} &&
Q_{1} \ar@{=}[d]^{\wr} \\
Q''_{2}\oplus Q'_{3}  \ar[rr]_{\begin{pmatrix}\id& 0 \\0 & 0 \end{pmatrix}}  &&
Q''_{2}\oplus Q'_{2}  \ar[rr]_{\begin{pmatrix}0 & 0 \\0 & \id \end{pmatrix}}  &&
Q''_{1}\oplus Q'_{2}  
}\end{xymatrix}
\end{equation}
Moreover $Q''_{2}$ is a direct summand of $N_{3}$ 
and $Q'_{2}$ is a direct summand  of $N_{2}$, and there are  the following diagrams
\[
\begin{split}
(i) \begin{xymatrix}
{
Q_{3} \ar[rr]^{f_{3}} \ar@{=}[d] && N_{3} \ar@{=}[d]\\
Q''_{2} \oplus Q'_{3}  \ar[rr]_{ \begin{pmatrix} \id_{Q''_{2}} & 0 \\ 0 & \ast \end{pmatrix}} 
 && Q''_{2} \oplus N'_{3},  
}
\end{xymatrix}
(ii) 
\begin{xymatrix}
{
N_{3} \ar[rr]^{g_{3} } \ar@{=}[d] && Q_{2} \ar@{=}[d]\\
Q''_{2} \oplus N'_{3}  \ar[rr]_{ \begin{pmatrix} \id_{Q''_{2}} & 0 \\ 0 & 0 \end{pmatrix}} 
 && Q''_{2} \oplus Q'_{2},  
}
\end{xymatrix}
\\
(iii)
 \begin{xymatrix}
{
Q_{2} \ar[rr]^{f_{2}} \ar@{=}[d] && N_{2} \ar@{=}[d]\\
Q''_{2} \oplus Q'_{2} \ar[rr]_{ \begin{pmatrix} 0 & 0 \\ 0 & \id_{Q'_{2}} \end{pmatrix}} 
 &&  N'_{2} \oplus Q'_{2},   
}
\end{xymatrix}
(iv) 
 \begin{xymatrix}
{
N_{2} \ar[rr]^{g_{2} } \ar@{=}[d] && Q_{1} \ar@{=}[d]\\
N'_{2} \oplus Q'_{2}  
\ar[rr]_{ \begin{pmatrix}  \ast & 0 \\ 0 &  \id_{Q'_{2}} \end{pmatrix}} 
 &&  Q'_{1} \oplus Q'_{2}.   
}
\end{xymatrix} 
\end{split}\] 
Thus, the exact triangle
$N_{3} \xrightarrow{g_{3}}  Q_{2} \xrightarrow{f_{2}} N_{2} \to $ is a direct sum of the 
three exact triangles
\[
\begin{split}
& 0 \to Q''_{2} \xrightarrow{\id} Q'_{2} \to,  \\ 
& Q''_{2} \xrightarrow{\id } Q''_{2} \to 0 \to,  \\
& N'_{3} \to 0 \to N'_{2} \to.  
\end{split}
\]
In particular $N'_{3} \cong N'_{2}[-1]$. 
\end{lemma}

\begin{proof}
First we prove the splitting \eqref{201804052136}. 
It follows from the exactness of \eqref{201901041818} for $N = \Cok \tuH^{0}(\delta_{3})$  that the complex $\tuH^{0}(Q_{3}) \xrightarrow{\tuH^{0}(\delta_{3})} \tuH^{0}(Q_{2}) \xrightarrow{\tuH^{0}(\delta_{2})} \tuH^{0}(Q_{1})$ is exact and hence 
it is a truncated projective resolution of $\Cok \tuH^{0}(\delta_{2})$. 
The cohomology of the   complex \eqref{201901041818} computes $\Ext^{1}(\Cok \tuH^{0}(\delta_{2}), N)$. 
Thus, we conclude that the $H^{0}$-module  $\Cok \tuH^{0}(\delta_{2})$ is projective 
and the complex $\tuH^{0}(Q_{3}) \xrightarrow{\tuH^{0}(\delta_{3})} \tuH^{0}(Q_{2})  \xrightarrow{\tuH^{0}(\delta_{2})} \tuH^{0}(Q_{1})$ splits. 

By Lemma \ref{basic of cP lemma}.(3), we have an equivalence $\cP \cong \Proj H^{0}$. 
Therefore we have the splitting \eqref{201804052136} as desired.

Let $\phi: Q_{3} \to Q''_{2}$ and $\phi': Q''_{2} \to Q_{3}$  be  
the retraction and the section of the left direct sum decomposition of \eqref{201804052136}. 
Let $\psi: Q_{2} \to Q''_{2}$ and  $\psi': Q''_{2} \to Q_{2}$ 
 be the retraction  and the  section of 
of the middle direct sum decomposition  \eqref{201804052136}. 
Note that we have $\delta_{3} = \psi'\phi$. 
\[
\begin{xymatrix}{
Q_{3} \ar[rr]^{\delta_{3}} \ar[d]_{\begin{pmatrix} \phi \\ * \end{pmatrix}} &&
Q_{2}  \\
Q''_{2}\oplus Q'_{3}  \ar[rr]_{\begin{pmatrix}\id& 0 \\0 & 0 \end{pmatrix}}  &&
Q''_{2}\oplus Q'_{2}  \ar[u]_{\begin{pmatrix} \psi'& * \end{pmatrix}} 
}\end{xymatrix}
\]

Then $\psi g_{3}: N_{3} \to Q''_{2}$ is a split epimorphism with a section $f_{3}\phi': Q''_{2} \to N_{3}$. 
Moreover, 
if we set $N'_{3} =\Ker \psi g_{3}$, then  the morphisms $f_{3}, g_{3} $ are of the following forms
\[
\begin{xymatrix}
{
Q_{3} \ar[rr]^{f_{3}} \ar@{=}[d] && N_{3} \ar@{=}[d]\\
Q''_{2} \oplus Q'_{3}  \ar[rr]_{ \begin{pmatrix} \id_{Q''_{2}} & 0 \\ 0 & f_{3}| \end{pmatrix}} 
 && Q''_{2} \oplus N'_{3},  
}
\end{xymatrix}
\begin{xymatrix}
{
N_{3} \ar[rr]^{g_{3}} \ar@{=}[d] &&  Q_{2} \ar@{=}[d]\\
Q''_{2} \oplus N'_{3}   \ar[rr]_{ \begin{pmatrix} \id_{Q''_{2}} & 0 \\ 0 & g_{3}|  \end{pmatrix}} 
 && Q''_{2} \oplus Q'_{2}
}
\end{xymatrix}
\]
where $f_{3}|, g_{3}|$ denote the restriction morphisms. 
The first diagram gives  $(i)$.
The second one gives $(ii)$ except the vanishing of $(2,2)$-component. i.e., $g_{3}|= 0$. 
We denote $\alpha:=g_{3}|$

In a similar way, we obtain $(iv)$ and 
$(iii)$ except the vanishing of $(1,1)$-component, 
which we denote by $\beta$. 

Observe that 
the exact triangle $N_{3} \xrightarrow{g_{3}} Q_{2} \xrightarrow{f_{2}} N_{2} \to $ is 
a direct sum of two exact sequences 
\[
 Q''_{2} \xrightarrow{\id} Q''_{2} \xrightarrow{\beta} N'_{2} \to, \ \ \ \ 
 N'_{3} \xrightarrow{\alpha} Q'_{2} \xrightarrow{\id} Q'_{2} \to. \] 
 Therefore we conclude that $\alpha = 0, \beta= 0$. 
\end{proof}

\begin{proof}[Proof of Theorem \ref{sppj resolution theorem lemma II}]
The implication (2) $\Rightarrow$ (3) is clear. 

We prove the implication (3) $\Rightarrow$ (1). 
From the assumption, we deduce that $M$ belongs to 
$\cP[-\sup P_{0}] \ast \cP[-\sup P_{1} +1] \ast \cdots \ast \cP[-\sup P_{e} +e]$. 
Thus by Corollary \ref{20171219208}, we conclude that 
$\pd M \leq e - \sup P_{e} + \sup P_{0} \leq d$.

We prove the implication (1) $\Rightarrow$ (2). 
If for $i \geq 1$ we have $\pd M_{i-1} > 0$, 
then 
 by Corollary \ref{20171219205}, we must have $\pd M_{j-1} > 0 $ for $j \leq i$. 
Therefore by Lemma \ref{20171219228}
$\pd M_{j} = \pd M_{j-1} - 1 -  \sup M_{j-1} + \sup M_{j}$ for $j \leq i$. 
Thus, we have 
\begin{equation}\label{20171219226}
\pd M_{i} + i + \sup M - \sup M_{i} = \pd  M
\end{equation}
 for $i$ such that $\pd M_{i -1} > 0$. 
Since $i + \sup M - \sup M_{i} \geq i$, 
the set $\{ i \geq 1 \mid \pd M_{i-1} > 0\}$ is finite. 
If we set $e:= \max \{ i \geq 1 \mid \pd M_{i-1} > 0\}$, then $\pd M_{e} = 0$ 
and hence $M \in \cP[-\sup M_{e}]$ by Lemma \ref{20171219205II}.  
From the equation (\ref{20171219226}) and $\pd M \leq d$, we deduce the desired inequality.

The implication (5) $\Rightarrow$ (1) is proved in Corollary \ref{20171219208}. 
The implications (2) $\Rightarrow$ (5), (1) $\Rightarrow$ (4) are clear. 

It remains to prove the implication (4) $\Rightarrow$ (1).

\begin{claim}\label{201804052018}
$M$ has a sppj resolution of finite length. 
\end{claim}

We postpone the proof of the claim.  
Assume that $M$ has a sppj resolution of length $\ell < \infty$. 
We prove that if $N \in \sfD^{[a,b]}(R)$, then $F(N) \in \sfD^{[a+t_{0}, b + t_{0} + d]}(\kk)$. 
First, the case $a, b \in \ZZ$ can be proved by induction on $b -a$. 
Next we deal with the case $a = -\infty$ and $b \in \ZZ$. 
Since $\Hom(\cP, \sfD^{< 0}(R)) =0 $, we have $F(\sigma^{< c}N)=0 $ for $c= -t_{\ell} -\ell$. 
Thus, $F(N) = F(\sigma^{\geq e}N)$ and the problem is reduced to the first case. 
Finally we deal with the case $a\in \ZZ$ and $b = \infty$. 
Since $\Hom(\cP, \sfD^{> 0}(R)) = 0$, 
we have $F(\sigma^{>- t_{0}}N) = 0$. 
Thus $F(N) \cong F( \sigma^{\leq - t_{0}} N) $ and the problem is reduced to the first case. 
To finish the proof we only have to show the claim. 

\begin{proof}[Proof of Claim \ref{201804052018}]
By shifting the degree, we may assume that $\sup M = 0 $. 
Let $P_{\bullet}$ be sppj resolutin of $M$. 
We set $t_{i} :=  -\sup P_{i}$. 
We take $\ell$ to be a natural number such that $ \ell + t_{\ell} > d-\sup M$. 
Then for $N \in \Mod H^{0}$ and $j \geq \ell$, we have $\Hom(M, N[j + t_{j}]) = 0$.


The situation is divided into the following two cases. 

(A) $t_{j} =t_{\ell}$ for $j \geq \ell$. 

(B) there exists $k > \ell$ such that 
$t_{\ell  } = t_{\ell - 1} = \cdots = t_{k-1} \neq t_{k}$.

First we deal with the case (A). 
We prove that $M_{\ell +1}$ belongs to $\cP[-\sup M_{\ell+ 1}]$ 
and hence that  $M$ has a sppj resolution of length $\ell +1$. 
Thanks to Lemma \ref{sppj resolution lemma}, 
we can apply (the shifted version of)  
Lemma \ref{201711171950} 
to $P_{j+1} \to P_{j} \to P_{j-1}$ for $j > \ell$. 

The sppj morphism $f_{j} : P_{j} \to M_{j}$  
is of the form
\[
\begin{pmatrix}
0 & 0 \\ 
0 & \id
\end{pmatrix}: 
P''_{j}\oplus P'_{j} 
\to M'_{j} \oplus P'_{j}. 
\]
and $M'_{j+1} \cong M'_{j}[-1]$. 
We note that since $\tuH^{- t_{\ell}}(f_{j})$ is surjective, we have $\tuH^{-t_{\ell}}(M'_{j}) = 0$. 

Therefore  for $i \geq 0$,  
\[
\tuH^{-t_{\ell} -i}(M'_{\ell+ 1}) = \tuH^{-t_{\ell}}(M'_{\ell+ 1}[-i]) = \tuH^{-t_{\ell}}(M'_{\ell+1+i}) =0.  
\]
Hence $\tuH^{\leq -t_{\ell}}(M'_{\ell +1}) = 0$. 
On the other hand, since $\sup M_{\ell +1} = -t_{\ell}$,  we have $\tuH^{>-t_{\ell}}(M_{\ell + 1}) = 0$. Thus, $\tuH(M'_{\ell + 1}) = 0$.  We conclude that  $M'_{\ell + 1} = 0$. 
 and $M_{\ell + 1}  \cong P'_{\ell +1}$ as desired.  

 Next we deal with the case (B). 
 We prove that the morphism $f_{k-1}: P_{k-1} \to M_{k-1}$ is an isomorphism 
and hence that $M$ has a sppj resolution of length $k -1$. 

First we assume that 
 $t_{k-1} \neq t_{k} \neq t_{k +1}$. 
Then $P_{k} = 0$  by Lemma \ref{sppj resolution lemma}. 
Thus in this case, $M_{k} = 0$. 
Hence $f_{k-1}: P_{k-1} \to M_{k-1}$ is an isomorphism.

Next we assume that 
 $t_{k -1} \neq t_{k} = t_{k +1}$. 
Then by Lemma \ref{sppj resolution lemma}  
the morphism $\delta_{k+1} : P_{k+1} \to P_{k}$ 
is a split epimorphism. 
Consequently,  
the morphism $g_{k+1}: M_{k+1} \to P_{k}$ is a split epimorphism. 
Therefore, 
there exists $L_{k+1}$ 
and an isomorphism $M_{k+1} \cong L_{k+1} \oplus P_{k}$ 
under which  $g_{k+1}$ corresponds to the canonical projection. 
Thus we conclude that  $ f_{k} = 0$ and 
by the definition of a sppj morphism  $M_{k} = 0$. 
Hence $f_{k-1}: P_{k-1} \to M_{k-1}$ is an isomorphism.  
\end{proof}

This completes the proof of Theorem \ref{sppj resolution theorem lemma II}. 
\end{proof}

The following is the main theorem of this section.

\begin{theorem}\label{sppj resolution theorem}
Let $M \in \sfD^{< \infty}(R)$ and $d \in \NN$ a natural number. 
Then  
the following conditions are equivalent 

\begin{enumerate}[(1)]
\item 
$\pd  M  = d$. 

\item 
For any sppj resolution $P_{\bullet}$, 
there exists a natural number $e \in \NN$ 
which satisfying the following properties 

\begin{enumerate}[(a)]
\item $M_{e} \in \cP[-\sup M_{e}]$. 

\item 
$d = e+ \sup P_{0} -\sup M_{e}$. 

\item 
$g_{e}$ is not a split-monomorphism. 
\end{enumerate}

\item 
$M$ has sppj resolution $P_{\bullet}$ of length $e$ 
which satisfies the following properties. 
\begin{enumerate}[(a)]
\item 
$d = e+ \sup P_{0} -\sup P_{e}$. 

\item 
$\delta_{e}$ is not a split-monomorphism. 
\end{enumerate}

\item 
 The functor $F= \RHom(M, -)$ sends the standard heart $\Mod H^{0}$ 
to $\sfD^{[-\sup M , d -\sup M]}(R)$ 
and there exists $N \in \Mod H^{0}$ such that $\tuH^{d  -\sup M}(F(N)) \neq 0$.

\item 
$d$ is the smallest number 
which satisfies  
\[
M \in \cP[-\sup M ] \ast \cP[-\sup M+1] \ast \cdots \ast \cP[ -\sup M +d]. 
\]

\end{enumerate}

\end{theorem}

\begin{proof}
The implication (2) $\Rightarrow$ (3) is clear. 

We prove the implication (1) $\Rightarrow$ (2). 
In the proof of  Theorem \ref{sppj resolution theorem lemma II}, 
we showed that if we set $e = \max \{ i \mid \pd M_{i-1} > 0\}$, 
then $M_{e} \in \cP[-\sup M_{e}]$ and $\pd M = e + \sup M - \sup M_{e}$. 
If the morphism $g_{e}: M_{e} \to P_{e-1}$ is a split-monomorphism, 
then $M_{e-1}$ is a direct summand of $P_{e}$ and $\pd M_{e -1} = 0$. 
This contradict to the definition of $e$.

We prove the implication (3) $\Rightarrow$ (1). 
We remark that the condition that $\delta_{e}$ is not a split-monomorphism 
implies that $P_{e} \neq 0$. 
By Theorem \ref{sppj resolution theorem lemma II}, it is enough to prove that 
there exists $N \in \Mod H^{0}$ such that 
$\Hom(M, N[d -\sup M]) \neq 0$. 

Assume that  $\sup P_{e -1} \neq \sup P_{e}$. 
Then by Lemma \ref{sppj resolution lemma} for $N \in \Mod H^{0}$, 
\[
\Hom(M, N[d -\sup M ]) \cong  \Hom(\tuH^{\sup}(P_{e}) ,N).  
\]
Thus, $N = \tuH^{\sup}(P_{e})$ has the desired property.

Assume that  $\sup P_{e -1} =  \sup P_{e}$. 
Then by Lemma \ref{sppj resolution lemma} for $N \in \Mod H^{0}$, 
\[
\Hom(M, N[d -\sup M ]) \cong  \Coker[\Hom(\tuH^{\sup}(P_{e-1}) ,N) \to \Hom(\tuH^{\sup}(P_{e}) ,N)].  
\]
Since $\delta_{e}$ is not a split-monomorphism,  
$N = \tuH^{\sup}(P_{e})$ has the desired property. 

The equivalence (4) (resp. (5)) to the other conditions  follows from Theorem \ref{sppj resolution theorem lemma II}. 
\end{proof}

The length $e$ of a sppj resolution $P_{\bullet}$ of a DG-module $M$ possibly has several value.

\begin{example}
Let $n \in \NN$ and $M_{(n)} := R \oplus R[n]$. Then, it can be directly checked that $\pd M_{(n)} = n$. 
The DG-module $M_{(n)}$ has a sppj resolution of length $2$. 
\[
P_{\bullet} :  R[n-1] \xrightarrow{ 0 } R \xrightarrow{ \begin{pmatrix} 1 \\ 0 \end{pmatrix} } M_{(n)}. 
\]
On the other hand, consider the exact triangles   
\[
E_{m}: M_{(m-1)} \to R^{\oplus 2} \xrightarrow{ \begin{pmatrix} 1  & 0 \\ 0 & 0 \end{pmatrix} } M_{(m)}.  
\]
Since $M_{(0)} =R^{\oplus 2} $, 
splicing $\{E_{m}\}_{m =1}^{n}$ 
we obtain a sppj resolution of length $n$ 
\[
 M_{(0)} \to R^{\oplus 2} \to \cdots \to R^{\oplus 2} \to M_{(n)}. 
\]
\end{example}

\subsection{The subcategory of DG-modules of finite projective dimension}

Theorem \ref{sppj resolution theorem} has the following consequences.

We denote by $\sfD(R)_{\textup{fpd}}$ the full subcategory consisting of $M$ having finite projective dimension. 

\begin{proposition}
\[
\sfD(R)_{\textup{fpd}} = \thick \cP = \bigcup \cP[a]\ast \cP[a +1] \ast \cdots \ast \cP[b] 
\]
where $a,b$ run all the pairs of integers such that $a \leq b$. 
\end{proposition}

We denote by $\sfD_{\mod H^{0}}(R)$ the full subcategory 
consisting of $M$ such that $\tuH^{i}(M)$ is a finitely generated graded  $H^{0}$-module for $i\in \ZZ$. 
We set 
\[
\sfD_{\mod H^{0}}(R)_{\textup{fpd}} = \sfD_{\mod H^{0}}(R) \cap \sfD(R)_{\textup{fpd}}, \ 
\cP_{\mod H^{0}}  = \sfD_{\mod H^{0}}(R) \cap \cP. 
\]
Note that $\cP_{\mod H^{0}} = \add R$, that is,  objects of $\cP_{\mod H^{0}}$ are precisely   direct summands  of  finite direct sums of $R$. 
 Therefore $\thick \cP_{\mod H^{0}} = \thick R$ is nothing but a perfect derived category $\Perf R$ of $R$.

We recall the notion of piecewise Noetherian DG-algebra, 
which play a role of Noetherian algebras in theory of ordinary rings. 

\begin{definition}
We call a DG-algebra $R$ 
\textit{right piecewise Noetherian} 
if $H^{0}$ is right Noetherian and $H^{-i}$ is finitely generated as a right $H^{0}$-module 
for $i \geq 0$. 
\end{definition}

The name is taken from \cite{Avramov-Halperin}. 
The same notion is called right cohomological pseudo-Noetherian in \cite{Yekutieli} 
and right Noetherian in \cite{Shaul:Homological, Shaul:Injective}.

\begin{proposition}\label{201901252356}
If $R$ is piecewise Noetherian, 
then 
\[
\sfD_{\mod H^{0}}( R)_{\textup{fpd}} = \Perf R =
\bigcup \cP_{\mod H^{0}}[a]\ast \cP_{\mod H^{0}}[a +1] \ast \cdots \ast \cP_{\mod H^{0}}[b] 
\]
where $a,b$ run all the pairs of integers such that $a \leq b$. 
\end{proposition}

\begin{proof}
It is clear that the first one contains the second one 
and that the second one contains the third one. 

For an object $M \in \sfD_{\mod H^{0}}(R)$, 
we can construct a sppj resolution $P_{\bullet}$ of $M$ 
such that $P_{i}$ belongs to $\cP_{\mod H^{0}}$. 
Thus if more over $\pd M < \infty$, it belongs to the third one. 
\end{proof}

\subsection{Tensor product}

For $P \in \cP$ and $L \in \Mod (H^{0})^{\op}$, 
\[
\tuH^{n}(P \lotimes_{R} L) = 
\begin{cases}
 \tuH^{0}(P) \otimes_{H^{0}} L & n = 0,\\
  0 & n \neq 0. 
 \end{cases}
\]
In a similar way of Lemma \ref{sppj resolution lemma}, 
we can prove the following lemma. 

\begin{lemma}\label{sppj resolution tensor lemma}
Let $L \in \Mod (H^{0})^{\op}, M \in \sfD^{< \infty}(R)$ and $P^{\bullet}$  a sppj resolution of $M$. 
We denote the complexes below by $Y_{i}, Y'_{i}$. 
\[
\begin{split}
Y_{i} : 
\tuH^{\sup P_{i}}(P_{i+1}\lotimes_{R} L) \to 
\tuH^{\sup P_{i}}(P_{i}\lotimes_{R} L) \to
\tuH^{\sup P_{i}}(P_{i-1}\lotimes_{R} L)  \\ 
Y'_{i} : 
\tuH^{\sup P_{i}}(P_{i+1}) \otimes_{H^{0}} L \to
\tuH^{\sup P_{i}}(P_{i}) \otimes_{H^{0}} L \to
\tuH^{\sup P_{i}}(P_{i-1}) \otimes_{H^{0}} L 
\end{split}
\] 
Then, 
\[
\tuH^{n}(M \lotimes_{R} L) =
\begin{cases} 
0 & n \neq -i +\sup P_{i} \textup{ for any } i \geq 0 \\
\tuH(Y_{i}) & n = -i +\sup P_{i} \textup{ for some } i \geq 0 
\end{cases}
\]

Moreover, in the case $n =- i + \sup P_{i}$,  we have  
\[
\tuH(Y_{i}) =
\begin{cases}
\tuH^{\sup P_{i}}(P_{i}) \otimes_{H^{0}} L  & \sup P_{i -1} \neq \sup P_{i} \neq \sup P_{i+1}, \\
\Coker[\tuH^{\sup P_{i}}(P_{i+1}) \otimes_{H^{0}} L \to
\tuH^{\sup P_{i}}(P_{i}) \otimes_{H^{0}} L ],  &
 \sup P_{i-1} \neq \sup P_{i} = \sup P_{i+1}, \\ 
 \Ker[
\tuH^{\sup P_{i}}(P_{i}) \otimes_{H^{0}} L \to
\tuH^{\sup P_{i}}(P_{i-1}) \otimes_{H^{0}} L 
 ], & 
 \sup P_{i -1} = \sup P_{i} \neq \sup P_{i+1}, \\ 
\tuH(Y'_{i}), & 
 \sup P_{i -1} = \sup P_{i} = \sup P_{i +1}. 
\end{cases}  
\]
\end{lemma}

\subsection{Minimal sppj resolution}

From Lemma \ref{sppj resolution lemma} we deduce the following corollary.

\begin{corollary}\label{sppj resolution corollary 1}
Assume that $M \in \sfD^{< \infty}(R)$ admits a minimal sppj resolution $P_{\bullet}$. 
Then for a simple $H^{0}$-module $S$ we have 
\[
\Hom(M,S[n]) 
= 
\begin{cases}
0 & n \neq i -\sup P_{i} \textup{ for any } i \geq 0,  \\
\Hom(\tuH^{\sup}(P_{i}), S) & n = i - \sup P_{i} \textup{ for some } i \geq 0.
\end{cases}
\]
\end{corollary}

From Lemma \ref{sppj resolution tensor lemma} we deduce the following corollary.

\begin{corollary}\label{sppj resolution corollary 2}
Assume that $M \in \sfD^{< \infty}(R)$ admits a minimal sppj resolution $P_{\bullet}$. 
Then for a simple $(H^{0})^{\op}$-module $T$ we have 
\[
\tuH^{n}(M\lotimes_{R} T) 
= 
\begin{cases}
0 & n \neq - i +\sup P_{i} \textup{ for any } i \geq 0,  \\
\tuH^{\sup}(P_{i})\otimes_{H^{0}} T & n = -i + \sup P_{i} \textup{ for some } i \geq 0. 
\end{cases}
\]
\end{corollary}

\section{Inf-injective resolutions of DG-modules}\label{ifij resolution}

\subsection{Injective dimension of $M\in \sfD(R)$ after Yekutieli}

We recall the definition of the injective dimension of $M \in \sfD(R)$ 
introduced by Yekutieli. 

\begin{definition}[{\cite[Definition 2.4]{Yekutieli}}]
Let  $a \leq b \in \{ -\infty \} \cup \ZZ \cup \{\infty\}$. 

\begin{enumerate}[(1)] 
\item 
An object $M \in \sfD(R)$ 
is said 
to have \textit{injective concentration} $[a,b]$ 
if 
such that the functor  $F = \RHom_{R}(-, M)$ 
sends $\sfD^{[m,n]}(R)$ to 
$\sfD^{[a-n,b -m ]}(\kk)$ 
for any $m \leq n \in \{-\infty\} \cup \ZZ \cup \{\infty\}$.
\[
F(\sfD^{[m,n]}(R)) \subset \sfD^{[a-n, b-m]}(\kk).
\]

\item 
An object $M \in \sfD(R)$ 
is said 
to have \textit{strict injective concentration} $[a,b]$
if it has injective concentration $[a,b]$ 
and does't have injective concentration $[c,d]$ 
such that $[c,d] \subsetneq [a,b]$. 

\item 
An object $M \in \sfD(R)$ 
is said 
to have injective dimension $d \in \NN$ 
if it has strict injective concentration $[a,b]$ for $a,b\in \ZZ$ 
such that $d= b-a$.

In the case where, $M$ does't have a finite interval as  injective concentration, 
it is said to have infinite injective dimension. 

We denote the injective  dimension by $\id M$.
\end{enumerate}

\end{definition}

\begin{remark}
Let $R$ be an ordinary algebra. 
Avramov-Foxby \cite{Avramov-Foxby:Homological dimensions} introduced 
another   injective dimension for a complex $M \in \sfC(R)$. 
If we denote by $\textup{AF}\injdim M$ the projective dimension of Avramov-Foxby, 
then it is easy to see that $\textup{AF}\injdim M = \injdim M +\inf M$. 
\end{remark}

The following lemma is the injective version of Lemma \ref{201711191911}.  

\begin{lemma}
If $M \in \sfD(R)$ has finite injective dimension, 
then it belongs to $\sfD^{>-\infty}(R)$. 
\end{lemma}

\begin{proof}
Since $R\in \sfD^{\leq 0}(R)$, 
the complex $M = \RHom(R, M)$ must belong $\sfD^{> -\infty}(\kk)$. 
\end{proof}

The following lemma is the injective version of Lemma \ref{201712182258}. 
We omit the proof, since it can be proved in a similar way.

\begin{lemma}
Let $M \in \sfD^{>-\infty}(R)$ and $ F:=\RHom(-,M)$. 
Then for any $m\leq n$, 
\[
F(\sfD^{[m,n]}(R)) \subset \sfD^{[-n+\inf M, \infty]}(\kk). 
\]
Moreover, there exists $N \in \sfD^{[m,n]}(R)$ 
such that $\tuH^{-n+\inf M}(F(N)) \neq 0$. 
\end{lemma}


We deduce the following useful corollaries. 

\begin{corollary}\label{injective dimension basics corolllary} 
Let $M \in \sfD^{> -\infty}(R)$ and $d \in \NN$. 
Then $\injdim M = d$ if and only if $M$ has a strict injective concentration $[\inf M,\inf M+d]$. 
\end{corollary}

\begin{corollary}\label{201804030002II}
Let $L \to M  \to N \to $  be an exact triangle in $\sfD(R)$. 
Then, we have 
\[
\injdim M + \inf M \leq \sup \{ \injdim L + \inf  L, \injdim N + \inf N\}.
\]
\end{corollary}

\subsection{The class $\cI$}\label{the class cI} 

The aim of Section \ref{the class cI} 
is to introduce the full subcategory $\cI \subset \sfD(R)$ 
which play the role of $\cP$ for inf-injective resolutions. 

\begin{definition}
Let $A = \bigoplus_{i \leq 0} A^{i}$ be a graded algebra. 
By $\Inj^{0} A \subset \GrMod A$, 
we denote the full subcategory 
consisting of injective graded $A$-modules $J$ 
cogenerated by degree $0$-part. 
More precisely, 
$J$ is assumed to 
satisfy the following conditions: 
\begin{enumerate}[(a)] 
\item $J$ is a graded injective $A$-module.

\item $J^{0}$ is an essential submodule of $J$.  

It might be worth noting that $J^{< 0} = 0$ follows from the second condition. 
\end{enumerate}
\end{definition}

A point is that $\Inj^{0} A$ is equivalent to the category $\Inj A^{0}$ 
of injective $A^{0}$-modules  as shown in Lemma below. 

\begin{lemma}\label{201710231648} 
Let $A = \bigoplus_{i\leq 0}A^{i}$ be a graded algebra. 
Then 
the following assertion is true. 
\begin{enumerate}[(1)] 
\item 
We have an adjoint pair 
\[
\phi : \GrMod A \rightleftarrows \Mod A^{0}: \psi
\]
defined by $\phi(M) := M^{0}$ and $\psi(L) := \Hom_{A^{0}}^{\bullet}(A, L)$. 

\item 
The adjoint pair $\phi \dashv \psi$ is restricted to give  an equivalence 
$\Inj^{0} A \cong \Inj A^{0}$. 
\end{enumerate}
\end{lemma}

\begin{proof}
(1) is  standard and so  is left to the readers. 

(2) 
We can check that $\phi$ and $\psi$ are restricted to the functors 
between $\Inj^{0} A$ and $\Inj A^{0}$. 
It is also easy to check that $\phi\circ \psi = \id_{\Inj A^{0}}$. 

For $J \in \Inj^{0} A$, we have a canonical map $f: J \to \Hom^{\bullet}_{A^{0}}(A, J^{0}) = \psi \circ \phi(J)$ 
which is the identity map $f^{0} = \id_{J^{0}}$ at degree $0$. 
This implies that $\Ker f \cap J^{0} = 0$, hence by the definition on $\Inj^{0} A$, $\Ker f = 0$. 
Thus $\psi \circ \phi(J) \cong J \oplus J'$ for some $J'$. 
However, $J' \cap \psi \circ \phi(J)^{0} = 0$, we conclude that $J' = 0$. 
This prove that $\psi \circ \phi  \cong \id_{\Inj^{0} A}$. 
\end{proof}

Since $R$ is a DG-algebra, 
we may equip  
$
\psi_{R}(K) := \Hom_{R^{0}}^{\bullet}(R, K)
$
 with the differential induced from that of $R$
and regard it as a DG-$R$-module. 



We denote by $\Inj^{0} R $ 
 the full subcategory of $\sfC(R)$ 
consisting of  the DG-$R$-modules of the forms 
$\psi_{R} (K) = \Hom_{R^{0}}^{\bullet}(R, K)$
for some $K \in \Inj R^{0}$.

\begin{definition}
By $\cI \subset \sfD(R)$, 
we denote the full subcategory consisting of (the quasi-isomorphism class of) $\psi_{R}(K)$ 
for $K \in \Inj R^{0}$. 
In the case we need to emphasize the DG-algebra $R$, we denote $\cI$ by $\cI(R)$. 
\end{definition} 

The properties  of $\cI$ which is used to construct inf-injective resolution  is summarized in 
the following Theorem, 
which is an injective version Lemma \ref{basic of cP lemma}.

\begin{theorem}\label{Shaul theorem}
The followings hold. 
\begin{enumerate}[(1)]
\item 
For $N \in \sfD(R)$ and $I \in \cI$, 
\[
\Hom(N,I[n] ) \cong \Hom_{\Mod H^{0}}(\tuH^{-n}(N), \tuH^{0}(I))
\]

\item 
For $N \in \Mod H^{0} $ and $I \in \cI$, 
\[
\Hom(N,I[n] ) \cong 
\begin{cases} 
\Hom_{\Mod H^{0}}(\tuH^{0}(N), \tuH^{0}(I)) & n =0, \\
0 & n \neq 0. 
\end{cases} 
\]

\item The cohomology functor $\tuH : \sfD(R) \to \GrMod H$ induces an equivalence 
\[
\tuH: \cI \xrightarrow{\cong} \Inj^{0} H.
\] 

\item 
Therefore, the $0$-th cohomology functor $\tuH^{0}: \sfD(R) \to \GrMod H$ induces an equivalence
\[
\tuH^{0}: \cI \xrightarrow{\cong} \Inj H^{0}.
\] 
\end{enumerate}
\end{theorem}

We need preparations. 
Let $\pi: R^{0} \to H^{0}$ be the canonical projection. 
We denote by $\pi^{!}$ the functor $\pi^{!}(M) := \Hom_{R^{0}}(H^{0}, M)$. 
\[
\pi^{!}: \Mod R^{0} \to \Mod H^{0}. 
\]
It is easy to see that the functor $\pi^{!}$ can be restricted to the functor 
$\pi^{!}: \Inj R^{0} \to \Inj H^{0}$.

\begin{lemma}\label{201710231815}
Let $K \in \Inj R^{0}$. 
Then for a DG-$R$-module $M$, we have isomorphisms 
\[
\begin{split}
(1)  \ \ \Hom_{R}^{\bullet}(M, \psi_{R}(K) ) & \cong \Hom_{R^{0}}^{\bullet}(M, K), \\  
(2)  \ \  \Hom_{\sfC(R)}(M,\psi_{R}(K)) &\cong  \tuZ^{0}(\Hom_{R}^{\bullet}(M, \psi_{R}(K))) \cong \Hom^{0}_{R^{0}}(M/\tuB(M), K), \\
(3)  \ \  \Hom_{\sfK(R)}(M,\psi_{R}(K)) &\cong  \tuH^{0}(\Hom_{R}^{\bullet}(M, \psi_{R}(K) )) \\
  & \cong \Hom^{0}_{H}(\tuH(M), \psi_{H} \pi^{!}(K) ) \\ 
  &\cong \Hom_{\Mod H^{0}}(\tuH^{0}(M), \pi^{!}K). 
\end{split}
\]
\end{lemma}

\begin{proof}
(1) We have  the adjunction isomorphism $\gamma$ below. 
\[
\gamma: \Hom_{R}^{\bullet}(M, \psi_{R}(K)) \xrightarrow{\cong} \Hom_{R^{0}}(M,K), \ \ \gamma(\phi)(m) := \phi(m)(1)
\]
It is an isomorphism of graded $\kk$-modules. 
We can check that $\gamma$ is compatible with the differentials of both sides 
and hence it gives an isomorphism in $\sfC(\kk)$.

(2) Since $K$ is regarded  as a DG-$R^{0}$-module with zero differentiall, 
for an homogeneous element $f \in \Hom_{R^{0}}^{\bullet}(M, K)$, 
we have 
$\partial(f) = (-1)^{|f|+ 1}f \circ \partial_{M}$. 
Therefore we have 
\[
\tuZ^{0} \Hom_{R^{0}}^{\bullet}(M, K)  \cong \Hom_{R^{0}}^{0}(M/\tuB M, K). 
\]

(3)
follows from the string of isomorphism 
\[
\begin{split}
\tuH^{0} \Hom_{R^{0}}^{\bullet}(M, K)  
&
\stackrel{\textup{(a)}}{\cong} \Hom_{R^{0}}^{0}(\tuH(M), K) 
\stackrel{\textup{(b)}}{\cong}  \Hom_{\Mod R^{0}}(\tuH^{0}(M), K)\\
&\stackrel{\textup{(c)}}{\cong} \Hom_{\Mod H^{0}}(\tuH^{0}(M), \pi^{!}K)
\stackrel{\textup{(d)}}{\cong} \Hom_{H}^{\bullet}(\tuH(M), \psi_{H} \pi^{!}(K))
\end{split}
\]
where $\stackrel{\textup{(a)}}{\cong}$ follows from the fact that  $K$ is injective $R^{0}$-module, 
$\stackrel{\textup{(b)}}{\cong}$
is a consequence of the fact that $K$ is concentrated at $0$-th degree, 
$\stackrel{\textup{(c)}}{\cong}$
 is deduced from the equation $B^{0} \tuH(M) = 0$ 
and 
$\stackrel{\textup{(d)}}{\cong}$
is nothing but the adjoint isomorphism. 
\end{proof}

The above lemma has two corollaries. 

\begin{corollary}\label{201711241612}
For $K \in \Inj R^{0}$, we have
\[
\begin{split}
&(1) \ \tuH \psi_{R} (K) \cong \psi_{H}\pi^{!}(K), \\
&(2) \ \tuH^{0} \psi_{R} (K) \cong \pi^{!}(K).
\end{split}
\]
\end{corollary}

\begin{proof}
(1) The first isomorphism is clear. 
Substituting $M = R$ to Lemma \ref{201710231815}.(3), we obtain  the second isomorphism. 

(2) is obtained from (1) by applying  $\phi_{H}$. 
\end{proof}

Recall that a DG-$R$-module $I$ is called \emph{DG-injective} if its underlying graded $R$-module is injective 
and the $\Hom$-complex $\Hom_{R}^{\bullet}(A, I) $ is acyclic for any acyclic DG-$R$-module $A$. 
An important property of DG-injective module $M$ is that 
for any  DG-$R$-module $M$, 
the $\Hom$-space $\Hom_{\sfK(R)}(M, I)$ in the homotopy category $\sfK(R)$ is 
isomorphic to $\Hom$-space $\Hom(M, I)$ in the derived category $\sfD(R)$ via the canonical map $
\Hom_{\sfK(R)}(M ,I) \xrightarrow{\cong} \Hom(M ,I)$  (see for example \cite{Keller:ddc, Positselski}).

We can immediately deduce the following corollary from Lemma \ref{201710231815}.(3).  

\begin{corollary}\label{201711241611}
$\psi_{R}(K)$ is a DG-injective DG-$R$-module. 
In particular, for $M \in \sfC(R)$ we have 
\[
\Hom(M,\psi_{R}(K)) \cong \Hom_{\sfK(R)}(M, \psi_{R}(K)). 
\]
\end{corollary}

The next lemma is the last preparation.

\begin{lemma}\label{201711241633}
$\pi^{!}: \Inj R^{0} \to \Inj H^{0}$ is essentially surjective. 
\end{lemma}

\begin{proof}
Let $J \in \Inj H^{0}$ 
and let $E =E_{R^{0}}(J)$ denotes the injective hull of $J$ as an $R^{0}$-module. 
Then, we prove that $J = \pi^{!}E$. 

It is clear $J \subset \pi^{!} E$. 
Let $x \in \pi^{!}E$. 
In other word, $x$ is an element of $E$ such that $B^{0} x = 0$. 
Then the submodule $J + Rx$ is regarded as an $H^{0}$-module. 
Since the extension $J \subset E$  is essential, 
so is the extension $J \subset J + Rx$. 
Thus, $J \subset J + Rx$ is an essential extension of $H^{0}$-modules. 
As  injective $H^{0}$-modules have a maximality of essential extensions (see e.g. \cite[18.11]{Anderson-Fuller}), 
we have $J = J +Rx$. 
In particular we have $x \in J$. 
This shows $\pi^{!} E \subset J$. 
\end{proof}

\begin{proof}[Proof of Theorem \ref{Shaul theorem}]
Let $I := \psi_{R}(K)$ for some $K \in \Inj^{0} R$. 
Then, combining Lemma \ref{201710231815}.(3), Corollary \ref{201711241612}.(2)  and Corollary \ref{201711241611}
we obtain an isomorphism 
\[
\begin{split}
\Hom(N,I)  &\cong \Hom_{\sfK(R)}(N, \phi_{R}(K)) \\
&  \cong \Hom_{\Mod H^{0}}(\tuH^{0}(N),\pi^{!}K ) \\
& \cong \Hom_{\Mod H^{0}}(\tuH^{0}(N), \tuH^{0}(I))
\end{split} 
\]
for $N \in \sfD(R)$. 
Considering shifts, we obtain (1). 

(2) immediately follows from (1). 

We prove (4). 
By (1) the functor $\tuH^{0}: \cI \to \Inj H^{0}$ is fully faithful. 

By Lemma \ref{201711241612}, the upper square of the following diagram is commutative. 
\[
\begin{xymatrix}{ 
\Inj R^{0} \ar[d]_{\psi_{R}} \ar[rr]^{\pi^{!}} && 
\Inj H^{0}  \\
\Inj^{0} R \ar[rr]^{\tuH} \ar[d]_{\textup{qis. class}} &&
\Inj^{0} H \ar[u]_{\phi_{H}}\\
\cI \ar[urr]_{\tuH}&&
}\end{xymatrix}
\]
The commutativity of the lower triangle is obvious. 
Thus the functor $H^{0} := \phi_{H} \tuH : \cI \to \Inj H^{0}$ 
is essentially surjective  by Lemma \ref{201711241633}. 
This complete the proof of (4). 

(3) follows from (4) by Lemma \ref{201710231648}. 
\end{proof}

\begin{proof}[Proof of Lemma \ref{201711191911}]
Let $F: =\RHom(M, -)$. 
By the assumption, there exists an integer $a$ such that 
$F(\sfD^{\geq 0}(R)) \subset \sfD^{\geq a}(\kk)$. 
Let  $J \in \Inj H^{0}$ and $I \in \cI$ be a unique object such that $\tuH^{0}(I) \cong J$. 
By  Theorem \ref{Shaul theorem}.(1), we have $\tuH^{n}(F(I)) \cong \Hom(M, I[n]) \cong \Hom(\tuH^{-n}(M), J)$. 
Therefore we have $\Hom(\tuH^{-n}(M), J) = 0$ for any $J \in \Inj H^{0}$ provided that $n < a$. 
Since  $\Mod H^{0}$ has injective cogenerator, it follows  that $\tuH^{>- a}(M) = 0$.
\end{proof}

\begin{remark}\label{Shaul remark}
In this remark we explain a relationship between this section and Shaul's paper \cite{Shaul:Injective}. 

In that paper Shaul defined the full subcategory  $\mathsf{Inj} R \subset \sfD(R)$ as a full subcategory consisting objects $I$ such that either $I =0$ or $\inf I = 0$ and $\injdim I = 0$. 
He showed that an object $I \in \mathsf{Inj } R$ satisfies the properties of Theorem \ref{Shaul theorem}. (1) and (2) and proved that 
the statement of Theorem \ref{Shaul theorem}.(3) holds for $\mathsf{Inj }R$.

As we mentioned before, Theorem \ref{Shaul theorem} is the dual of Lemma \ref{basic of cP lemma} which is stated for the full subcategory $\cP$. 
Moreover,   Corollary \ref{201901042037} which  tells $\cP$ consists of objects $P$ such that either $P= 0$ or $\sup P = 0$ and $\pd P = 0$,  
is a formal consequence of Lemma \ref{basic of cP lemma}. 
 Therefore, 
by the  dual argument we can prove  that if a  full subcategory $\cI \subset \sfD(R)$ satisfies properties given in Theorem \ref{Shaul theorem}, then  
$\cI$ consists of objects $I$ such that either $I = 0$ or $\inf I = 0$ and $\injdim I = 0$. 
This part can be said as the converse of Shaul's result.

In this paper we define the full subcategory $\cI$ as a full subcategory of $\sfD(R)$ consisting of DG-modules defined concretely. 
Then we prove that Theorem \ref{Shaul theorem} holds for $\cI$ and hence that $\cI$ consists of   either $I = 0$ or $\inf I = 0$ and $\injdim I = 0$. 

We will make use of the concrete construction of objects of $\cI$ in Section \ref{Bass-Papp section} and the subsequent work \cite{CDGA}. 
\end{remark}

\subsection{Inf-injective (ifij) resolutions}
We introduce the notion of an inf-injective (ifij)-resolution of $M \in \sfD^{> -\infty}(R)$ 
and show its basic properties. 
Since almost all the proofs are analogous to that for the similar statement of sppj resolution, 
we omit them.

\begin{definition}[ifij morphism and ifij resolution]\label{ifij morphism definition} 
Let $M \in \sfD^{>-\infty}(R), M \neq 0$. 

\begin{enumerate}
\item 
A ifij morphism $f: M \to I $ is a morphism in $\sfD(R)$ 
such that $I \in \cI[- \inf M]$ 
and the morphism $\tuH^{\inf M}(f)$ is injective. 

\item 
A ifij morphism $f: M \to I$ is called minimal 
if 
the morphism $\tuH^{\inf M}(f) $ is an injective envelope.

\item 
A ifij resolution $I_{\bullet}$ of $M$ is a sequence of exact triangles  for $i \leq 0$ 
with $M_{0} := M$ 
\[
M_{i} \xrightarrow{f_{i}} I_{i} \xrightarrow{g_{i}} M_{i-1}
\]
such that $f_{i}$ is ifij.

The following inequality folds
\[
\inf M_{i} = \inf I_{i} \leq \inf I_{i-1} = \inf M_{i-1} 
\]

For an ifij resolution $I_{\bullet}$ with the above notations, 
we set $\delta_{i} := f_{i-1} \circ g_{i}$. 
\[
\delta_{i} :I_{ i} \to I_{i-1}.
\]
Moreover we write 
\[M \to I_{0} \xrightarrow{\delta_{0}} 
I_{-1} \xrightarrow{\delta_{-1}}  \cdots \to I_{-i}  \xrightarrow{\delta_{-i}} I_{-i-1} \to \cdots. 
\]

\item 
A ifij resolution $I_{\bullet}$ is said to have length $e$ if 
$I_{-i} = 0$ for $i> e$ and 
$I_{-e} \neq 0$. 

\item 
A ifij resolution $I_{\bullet}$ is called minimal if $f_{i}$ is minimal 
for $i \leq 0$.
\end{enumerate}
\end{definition}

By Theorem \ref{Shaul theorem}, 
for any $M \in \sfD^{> -\infty}(R)$, 
there exists an ifij morphism $f:M \to I$.

We give basic properties of  ifij resolutions. 
We omit almost all the proof,  since these are done by the argument which is dual to that of the

\begin{lemma}
Let $M \in \sfD(R)$. 
Then, $\injdim M = 0$ if and only if $\inf M > - \infty$ and 
$M \in \cI[-\inf M]$. 
\end{lemma}

\begin{lemma}
Let $a ,b \in \ZZ$ such that $a \leq b$. 
Then for any object $M \in \cI[a] \ast \cI[a + 1] \ast \cdots \ast \cI[b]$, 
we have $\injdim M \leq b-a$.  
\end{lemma}

\begin{lemma}
Let $M \in \sfD^{>- \infty}(R) , M \neq 0$ 
and $M \xrightarrow{f}  I \xrightarrow{g} N$ an exact triangle with $f$ ifij. 
Assume that $ \injdim M \geq 1$. 
Then $\injdim N = \injdim  M -1 + \inf M - \inf N $. 
\end{lemma}

\begin{lemma}\label{ifij resolution lemma}
Let $N \in \mod H^{0}, M \in \sfD^{> -\infty}(R)$ and $I_{\bullet}$  an ifij resolution of $M$. 
We denote the complexes below by $Z_{i}, Z'_{i}$. 
\[
\begin{split}
Z_{-i} : \Hom(N, I_{-i+1}[\inf I_{-i}]) \to \Hom(N, I_{-i}[\inf I_{-i}] ) \to \Hom(N, I_{-i-1}[\inf I_{i}]) \\ 
Z'_{i} :\Hom(N, \tuH^{\inf}(I_{-i+1})) \to \Hom(N, \tuH^{\sup}(P_{-i})) \to \Hom(N, \tuH^{\sup}(P_{-i-1})) 
\end{split}
\] 
Then, 
\[
\Hom(N,M[n]) =
\begin{cases} 
0 & n \neq  i + \inf I_{-i} \textup{ for any } i \geq 0, \\
\tuH(Z_{-i}) & n = i + \inf I_{-i} \textup{ for some } i \geq 0. 
\end{cases}
\]

Moreover, in the case $n = i -  \inf I_{-i}$,  we have  
\[
\begin{split}
&\tuH(Z_{-i}) \\
&=
\begin{cases}
\Hom(N, \tuH^{\inf}(I_{-i})), & \inf I_{-i+1} \neq \inf I_{-i} \neq \inf I_{-i+1}, \\
\Ker[ \Hom(N, \tuH^{\inf}(I_{-i})) \to \Hom(N, \tuH^{\inf }(I_{-i-1})) ],  &
 \inf I_{-i+1}  \neq \inf I_{-i} = \inf I_{-i-1}, \\ 
 \Coker[\Hom(N, \tuH^{\inf}(I_{-i +1})) \to \Hom(N,  \tuH^{\inf}(I_{-i}) ) ], & 
 \inf I_{-i+1} = \inf I_{-i} \neq \inf I_{-i-1}, \\ 
\tuH(Z'_{i}), & 
 \inf I_{-i+1} = \inf I_{-i} = \inf I_{-i-1}. 
\end{cases}  
\end{split}
\]
\end{lemma}

We give an ifij version of Theorem \ref{sppj resolution theorem}

\begin{theorem}\label{ifij resolution theorem}
Let $M \in \sfD^{>-\infty}(R)$ and $d$ a natural number. 
Set $F :=\RHom(-,M) $. 
Then  
the following conditions are equivalent 

\begin{enumerate}[(1)]
\item 
$\injdim  M  = d$.

\item 
For any ifij resolution $I_{\bullet}$, 
there exists a natural number $e \in \NN$ 
which satisfying the following properties.  

\begin{enumerate}[(a)]
\item 
$M_{e} \in \cI[-\inf M_{e}]$. 

\item 
$d = e + \inf M_{-e} - \inf I_{0}$. 

\item 
$g_{-e}$ is not a split-epimorphism. 
\end{enumerate}

\item 
$M$ has ifij resolution $I_{\bullet}$ of length $e$ 
which satisfies the following properties.

\begin{enumerate}[(a)]
\item $d = e + \inf I_{-e} - \inf I_{0}$.

\item 
$\delta_{e}$ is not a split-epimorphism. 
\end{enumerate} 

\item 

The functor $F := \RHom(- ,M)$ sends the standard heart $\Mod H^{0}$ to 
$\sfD^{[\inf M, \inf M+d]}(\kk)$ 
and there exists 
$N \in \Mod H^{0}$ 
such that $\tuH^{\inf M +d}(F(N) ) \neq 0$.

\item 
The following conditions hold. 
\begin{enumerate}[(a)]
\item 
the functor $F=\RHom(-,M) $ 
sends finitely generated $H^{0}$-modules 
to $\sfD^{[\inf M, \inf M + d]}(\kk)$.

\item 
there exists a finitely generated $H^{0}$-module  $N \in \Mod H^{0}$
such that $\tuH^{\inf M + d}(F(N) ) \neq 0$.
\end{enumerate}

\item 
$d$ is the smallest number 
which satisfies  
\[
M \in \cI[-\inf M ] \ast \cI[-\inf M+1] \ast \cdots \ast \cI[-\inf M +d]. 
\]

\end{enumerate}

\end{theorem}

To deduce from the condition (5) to the other conditions, 
we need to use the following lemma. 

\begin{lemma}
A complex  $J_{\bullet}: J_{3} \xrightarrow{d_{3}}  J_{2} \xrightarrow{d_{2}} J_{1}$  of injective $H^{0}$-modules 
is exact and splits 
if and only if 
a complex below is exact for all finitely generated $H^{0}$-module $N$. 
\[
\Hom_{H^{0}}(N, J_{3}) \to \Hom_{H^{0}}(N, J_{2}) \to \Hom_{H^{0}}(N, J_{1})
\]
\end{lemma} 

\begin{proof}
``Only if" part is clear. We prove ``if" part. 
Substituting $N =H^{0}$ we see that the complex $J_{\bullet}$ is exact. 

We prove that $J_{\bullet}$ splits. 
Let $K :=\Ker d_{3}, L := \image d_{3} = \Ker d_{2}$. 
From the assumption, 
we can show that 
the induced map $\Hom_{H^{0}}(N, J_{3}) \to \Hom_{H^{0}}(N, L)$ is surjective 
for any  finitely generated $H^{0}$-modules $N$. 
This shows that $\Ext_{H^{0}}^{1}(N, K) = 0$ 
for any finitely generated $H^{0}$-module $X$. 
By Baer criterion (see \cite[18.3]{Anderson-Fuller}), $K$ is injective $H^{0}$-module.  
Thus the inclusion $K \hookrightarrow J_{3}$ splits 
and consequently $L$ is an injective $H^{0}$-module. 
Therefore the complex $J_{\bullet}$ splits. 
\end{proof}

\subsection{The subcategory of DG-modules of finite injective dimension}

As is the same with the sppj resolution, 
Theorem \ref{ifij resolution theorem} has the following consequences. 
We denote by $\sfD(R)_{\textup{fid}}$ the full subcategory consisting of $M$ having finite injective dimension. 

\begin{proposition}
\[
\sfD(R)_{\textup{fid}} = \thick \cI = \bigcup \cI[a]\ast \cI[a +1] \ast \cdots \ast \cI[b]. 
\]
where $a,b$ run all the pairs of integers such that $a \leq b$. 
\end{proposition}

We set $\sfD_{\mod H^{0}}(R)_{\textup{fid}} := \sfD_{\mod H^{0}}(R) \cap \sfD(R)_{\textup{fid}}$. 
Assume that  the base ring $\kk$ is a field and $R$ is locally finite dimensional, i.e., $\dim_{\kk} H^{i} < \infty$ for $i \in \ZZ$. 
Then the $\kk$-dual  complex  $R^{*} := \Hom_{\kk}(R, \kk)$ has canonical DG-$R$-module structure such that $\tuH(R^{*}) = H^{*}$ 
as graded $H$-modules. Therefore $R^{*}$ belongs to $\cI_{\mod H^{0}} := \cI \cap \sfD_{\mod H^{0}}(R)$. 
Since the $H^{0}$-module $(H^{0})^{*}$ is a injective cogenerator of $\mod H^{0}$, 
we have $\cI_{\mod H^{0}} = \add R^{*}$. 

By the same argument with Proposition \ref{201901252356} we obtain the following proposition. 

\begin{proposition}
Assume that  the base ring $\kk$ is a field and $R$ is locally finite dimensional. 
Then we have 
\[
\sfD_{\mod H^{0}}(R) = \thick R^{*} = \bigcup \cI_{\mod H^{0}}[a]\ast \cI_{\mod H^{0}}[a +1] \ast \cdots \ast \cI_{\mod H^{0}}[b]
\]
where $a,b$ run all the pairs of integers such that $a \leq b$. 
\end{proposition}

Jin \cite{Jin} introduced the notion of \textit{Gorenstein DG-algebra} 
and studied their representation theory. 
In \cite{Jin}  a Gorenstein DG-algebra is  defined to be a connective   DG-algebra 
which is proper, i.e., $\dim_{\kk} \sum_{i \in \ZZ} H^{i} < \infty$
 and satisfies the condition (J) $\Perf R = \thick R^{*}$. 
 We note that for a locally finite dimensional connective  DG-algebra $R$ over a field $\kk$, properness follows from the condition (J). 

On the other hand, in ordinary ring theory, 
Iwanaga-Gorenstein algebra is defined to be an (ordinary) algebra $A$ such that $\injdim_{A} A <\infty, \injdim_{A^{\op}} A < \infty$. 
Thus, there are obvious DG-generalization of the notion of a Iwanaga-Gorenstein algebra. 
In the next proposition, we show these two generalizations coincide with each other.

\begin{proposition}
Assume that the base ring $\kk$ is a field. 
Then for a locally finite dimensional DG-algebra $R$, the following conditions are equivalent. 

\begin{enumerate}[(1)] 
\item  $\Perf R = \thick R^{*}$.

\item $\injdim_{R} R < \infty, \injdim_{R^{\op}} R< \infty$. 

\end{enumerate}

\end{proposition}

\begin{proof}
(1) $\Rightarrow$ (2). Since $R \in \Perf R = \thick R^{*}$, we see $\injdim_{R} R< \infty$. 
Since the $\kk$-duality gives an contravariant equivalence $\sfD_{\mod H^{0}}(R)^{\op} \simeq \sfD_{\mod H^{0}}(R^{\op})$, 
we have $\Perf R^{op} = \thick R^{*}$ inside $\sfD(R^{\op})$. 
Thus in the same way as above we see $\injdim_{R^{\op}}R < \infty$. 

(2) $\Rightarrow$ (1). 
It follows from $\injdim_{R} R< \infty$ that $R \in \thick R^{*}$ and that $\Perf R \subset \thick R^{*}$. 
Similarly we have $\Perf R^{\op} \subset \thick R^{*}$ inside $\sfD(R^{\op})$. 
The $\kk$-duality sends the latter inclusion to $\thick R^{*} \subset  \Perf R$ inside $\sfD(R)$. 
Thus we see $\Perf R = \thick R^{*}$ as desired. 
\end{proof}

\subsection{Minimal ifij resolution}

From Lemma \ref{ifij resolution lemma} we deduce the following corollary. 
We note that 
since every module over an ordinary  algebra has an injective-hull, 
every object $M \in \sfD^{> - \infty}(R)$ admits a minimal ifij resolution. 

\begin{corollary} 
Let $M \in \sfD^{>-\infty}(R)$ and $I_{\bullet}$  a minimal ifij resolution of $M$. 
Then for a simple $H^{0}$-module $S$ we have 
\[
\Hom(S,M[n]) 
= 
\begin{cases}
0 & n \neq i +\inf I_{-i} \textup{ for any } i \geq 0, \\
\Hom(S, \tuH^{\inf}(I_{-i})) & n = i + \inf I_{-i} \textup{ for some } i \geq 0.
\end{cases}
\]
\end{corollary}

\subsection{The Bass-Papp theorem}\label{Bass-Papp section}

Recall that 
the Bass-Papp theorem claim that 
an ordinary algebra $A$ is right Noetherian 
if and only if any direct sum of injective (right) $A$-module is injective 
(see  \cite[Theorem 3.46]{Lam}). 
Shaul \cite{Shaul:Injective} gave a DG-version of it. 
We provide another proof by using the  method developed here.  

\begin{theorem}[{Shaul \cite{Shaul:Injective}}]\label{Bass-Papp theorem}
A connective DG-algebra $R$ is right piecewise Noetherian 
if and only if the class $\cI$ is closed under direct sums in $\cD(R)$. 
\end{theorem}

\begin{proof}
For $J \in \Inj H^{0}$, we set  
\[
G(J): = \psi_{R} (E_{R^{0}}(J) )= \Hom_{R^{0}}^{\bullet}(R, E_{R^{0}}(J))
\] 
where $E_{R^{0}}(J)$ denotes the injective-hull of $J$ as $R^{0}$-module. 
Then,  $G(J)$ belongs to $\cI$. 
Moreover, 
by Corollary \ref{201711241612}.(2) and the proof of Lemma \ref{201711241633}, 
we have  $\tuH^{0} G(J)  \cong J$ for $J \in \Inj H^{0}$.

We prove ``only if" part by showing that 
if  a small family $\{J_{\lambda}\}_{\lambda \in \Lambda}$ in $\Inj H^{0}$ is given, 
then the DG-$R$-module $\bigoplus_{\lambda \in \Lambda} G(J_{\lambda})$ is quasi-isomorphic to 
$G(J)$ for some $J \in \Inj H^{0}$. What we  actually prove is that  if we set  $J:=\bigoplus_{\lambda \in \Lambda} J_{\lambda}$, 
then there exists a quasi-isomorphism $f:\bigoplus_{\lambda \in \Lambda}G(J_{\lambda}) \xrightarrow{\sim} G(J)$. 

Let $\iota_{\lambda}: J_{\lambda} \to J$ be a canonical inclusion. 
It can be extended to a homomorphism $\underline{\iota}_{\lambda}: E_{R^{0}}(J_{\lambda}) \to E_{R^{0}}(J)$ 
between injective-hulls. 
Then, the morphism $f$ induced from the collection   $ \{G(\underline{\iota}_{\lambda})\}_{\lambda \in \Lambda}$ 
satisfies the desired property. 
\[
f:= (G(\underline{\iota}_{\lambda}))_{\lambda \in \Lambda}: \bigoplus_{\lambda \in \Lambda} G(J_{\lambda}) \to G(J)
\]
Indeed, taking the cohomology group of $G(\underline{\iota}_{\lambda})$, 
we obtain the morphism below
\[
\Hom(H,\iota_{\lambda}): \Hom_{H^{0}}(H, J_{\lambda}) \to \Hom_{H^{0}}(H,J). 
\]
Therefore, we have the equality of morphisms 
\[
\tuH(f) = (\Hom(H, \iota_{\lambda}))_{\lambda \in \Lambda}: 
\bigoplus_{\lambda \in \Lambda}\Hom_{H^{0}}(H, J_{\lambda}) 
\to \Hom_{H^{0}}(H,J).
\]
Finally, we observe that the right hand side is ensured to be an isomorphism 
by the assumption that $R$ is right piecewise Noetherian.

We prove ``if" part.
Let $\{J_{\lambda}\}_{\lambda \in \Lambda}$ be a small family in $\Inj H^{0}$. 
We set $J := \bigoplus_{\lambda \in \Lambda}J_{\lambda}$. 
By the assumption, 
the direct sum   $I := \bigoplus_{\lambda \in \Lambda} G(J_{\lambda})$ belongs to $\cI$. 
Thus  $J \cong \tuH^{0}(I)$ is 
an injective $H^{0}$-module by Theorem \ref{Shaul theorem}.  
We have shown that the class of  injective $H^{0}$-modules is closed under direct sum. 
Thus, by Bass-Papp theorem for ordinary rings, we conclude that $H^{0}$ is right Noetherian.

Since $J$ belongs to $\Inj H^{0}$, 
the canonical morphism $f: I \to G(J)$ defined as  above become an isomorphism 
after taking 
the $0$-th cohomology morphism $\tuH^{0}(f)$. 
Therefore, $f$ is an isomorphism by Theorem \ref{Shaul theorem}. 

Observe  that
the $i$-th cohomology morphism of $f$ is a canonical morphism  
induced from universal property of direct product
\[
\begin{split}
\bigoplus_{\lambda \in \Lambda} \Hom_{H^{0}}(H^{-i}, J_{\lambda}) 
\cong \tuH^{i}(I) 
\xrightarrow{\tuH^{i}(f) }
\tuH^{i}(G(J)) \cong 
\Hom_{H^{0}}(H^{-i}, \bigoplus_{\lambda \in \Lambda}J_{\lambda}). 
\end{split}
\]
It follows from 
Lemma \ref{noeth fg lemma} below that $H^{-i}$ is finitely generated. 
\end{proof}

\begin{lemma}\label{noeth fg lemma}
Let $A$ be a right noetherian algebra. 
Then an $A$-module $M$ is finitely generated if and only if 
for any family $\{J_{\lambda}\}_{\lambda \in \Lambda}$ of injective $A$-modules, 
the canonical morphism below is an isomorphism
\[ 
\gamma: \bigoplus_{\lambda \in \Lambda}\Hom_{A}(M, J_{\lambda}) \rightarrow \Hom_{A}(M, \bigoplus_{\lambda \in \Lambda} J_{\lambda}). 
\]
\end{lemma}

\begin{proof}
``only if" part is clear. We prove ``if" part by showing that 
if $M$ is infinitely generated, then 
the  canonical morphism $\gamma$ happens to become not surjective.

An infinitely generated $A$-module $M$ has  
 a strictly increasing sequence of submodules of $M$
\[
0 = : M_{0} \subsetneq M_{1} \subsetneq M_{2} \subsetneq \cdots M.
\]
Let $J_{i} := E_{A}(M_{i}/M_{i-1})$ be the injective-hull of $M_{i}/M_{i-1}$. 
We set $N := \bigcup_{i\geq 1} M_{i}, \ J = \bigoplus_{i \geq 1} J_{i}$. 
Then the composite morphism of the canonical morphisms  
$f_{i}: M_{i} \to M_{i}/M_{i-1} \to J_{i}$ 
extends to a morphism $f'_{i}: N \to J_{i}$ of $A$-modules.
Observe that the collection $\{f'_{i}\}$ induces a morphism $f: N \to J$. 
Since $A$ is right Noetherian, $J$ is injective. 
Therefore $f$ extends to $\tilde{f}: M \to J$. 

Let $\pi_{i}: J \to J_{i}$ be the canonical projection. 
Then, we have $\pi _{i}(\tilde{f}|_{M_{i}}) = f_{i}$. 
\[
f_{i} : M_{i} \hookrightarrow M \xrightarrow{\tilde{f}} J \xrightarrow{\pi_{i}} J_{i}. 
\]
Thus in particular $\pi_{i} \tilde{f} \neq 0$ for any $i \geq 1$. 
This shows that $\tilde{f}$ does not  belong to the image of $\gamma$. 
\end{proof}

\section{Sup-flat resolutions of DG-modules}\label{spft resolution}

\subsection{Flat dimension of $M\in \sfD(R)$ after Yekutieli}

We denote the opposite DG-algebra of $R$ by $R^{\op}$ 
and identify left DG-$R$-modules with (right) DG-$R^{\op}$-modules.

\begin{definition}[{\cite[Definition 2.4]{Yekutieli}}]
Let  $a \leq b \in \{ -\infty \} \cup \ZZ \cup \{\infty\}$. 

\begin{enumerate}[(1)] 
\item 
An object $M \in \sfD(R)$ 
is said 
to have \textit{flat concentration} $[a,b]$ 
if  the functor  $F = M \lotimes_{R}-$ 
sends $\sfD^{[m,n]}(R^{\op})$ to 
$\sfD^{[a+m,b+n ]}(\kk)$ 
for any $m \leq n \in \{-\infty\} \cup \ZZ \cup \{\infty\}$.
\[
F(\sfD^{[m,n]}(R^{\op})) \subset \sfD^{[a+m, b+n]}(\kk).
\]

\item 
An object $M \in \sfD(R)$ 
is said 
to have \textit{strict flat concentration} $[a,b]$
if it has flat concentration $[a,b]$ 
and does't have projective concentration $[c,d]$ such that $[c,d] \subsetneq [a,b]$. 

\item 
An object $M \in \sfD(R)$ 
is said 
to have flat dimension $d \in \NN$ 
if it has strict projective concentration $[a,b]$ for $a,b\in \ZZ$.
such that $d= b-a$.

In the case where, $M$ does't have a finite interval as  flat concentration, 
it is said to have infinite flat dimension.

We denote the flat dimension of $M$ by $\fd M$. 
\end{enumerate}

\end{definition}

\begin{remark}
Let $R$ be an ordinary algebra. 
Avramov-Foxby \cite{Avramov-Foxby:Homological dimensions} introduced 
another  flat dimension for a complex $M \in \sfC(R)$. 
If we denote by $\textup{AF}\fd M$ the projective dimension of Avramov-Foxby, 
then it is easy to see that $\textup{AF}\fd M = \fd M -\sup M$. 
\end{remark}

The following flat version of Lemma \ref{201711191911} follows from 
the isomorphism $M \lotimes_{R} R \cong M$. 

\begin{lemma}\label{201712191846}
If $M \in \sfD(R)$ has finite flat dimension, then it belongs to $\sfD^{< \infty}(R)$. 
\end{lemma}

The following lemma is deduced from the property of the derived tensor products.
\begin{lemma}\label{201712191857}
Let $M \in \sfD^{< \infty}(R)$ 
and $F:= M \lotimes_{R}-$. 
Then for all $m \leq n$, 
\[
F(\sfD^{[m,n]}(R^{\op})) \subset \sfD^{[-\infty, n + \sup M ]} (\kk)
\] 
and there is $N \in \sfD^{[m,n]}(R^{\op})$ such that $\tuH^{n + \sup M}(F(N) ) \neq 0$. 
\end{lemma}

\begin{proof}
The first statement follows from computation of $M \lotimes_{R} N$ by using DG-projective resolution of $M$ 
(see e.g. \cite{Positselski}). 

We set $b = \sup M$. Then $\tuH^{b}(M \lotimes_{R} H^{0} ) \cong \tuH^{b}(M) \otimes_{H^{0}} H^{0} \cong \tuH^{b}(M)$. 
Thus, $N := H^{0}[-n]$ has the desired property. 
\end{proof}

\subsection{The class $\cF$ and sup-flat (spft) resolution}

The class $\cF$ plays the role of flat modules in the usual flat resolutions  for  sup-flat resolutions.
However, an explicit description like $\cP, \cI$ has not been obtained at now. 

\begin{definition}
We denote by  $\cF \subset \sfD(R)$  
the full subcategory of those object $F \in \sfD(R)$ such that $\fd F = 0$. 
\end{definition}

The basic properties of $\cF$ are summarized in the lemma below. 

\begin{lemma}\label{basic of cF lemma} 
Let $F \in \cF$. 
Then the following statements hold. 
\begin{enumerate}[(1)]
\item 
Let $N \in \sfD(R^{\op})$. 
The canonical morphism $N \to \sigma^{\geq n} N$ induces an isomorphism below for $n \in \ZZ$. 
\[
\tuH^{n}(F \lotimes_{R} N) \to \tuH^{n}(F \lotimes_{R} \sigma^{\geq n} N)\]

\item 
Let $N \in \sfD(R^{\op})$. 
The canonical morphism $\sigma^{\leq n} N \to N$ induces an isomorphism below for $n \in \ZZ$. 
\[
\tuH^{n}(F \lotimes_{R} \sigma^{\leq n} N) \to \tuH^{n}(F \lotimes_{R}  N)
\]

\item 
For $N \in \Mod H^{0}$, we have 
\[
\tuH^{n}(F \lotimes_{R} N) =
\begin{cases}
\tuH^{0}(F) \otimes_{H^{0}} N & n = 0 \\ 
0  & n \neq 0
\end{cases}
\]

\item 
For $N \in \sfD(R^{\op})$, 
we have $\tuH^{n}(F \lotimes_{R} N) \cong \tuH^{0}(F) \otimes_{H^{0}}\tuH^{n}(N)$.

\item 
$\tuH^{0}(F)$ is a flat $H^{0}$-module 
and $\tuH(F) \cong \tuH^{0}(F) \otimes_{H^{0}} H$.  

\end{enumerate}
\end{lemma}

\begin{remark}
Lurie \cite[Section 7.2]{Lurie:HA} studied the class $\cF$ for  connective $\mathbb{E}_{1}$-algebras 
and characterized flat $\mathbb{E}_{1}$-modules by the condition (5) of above lemma  in \cite[Theorem 7.2.2.15]{Lurie:HA}. 
\end{remark}

\begin{proof}
(1) Since $F \lotimes_{R} \sigma^{< n} N$ belongs to $\sfD^{< n}(R)$, 
we have $\tuH^{i}( F \lotimes_{R} \sigma^{< n} N) = 0$ for $i = n, n+1$. 
Now the desired isomorphism is derived from the canonical exact triangle 
\[
\sigma^{< n} N \to N \to \sigma^{\geq n}N \to. 
\]

(2) is proved in a similar way of (1). 

(3) The case $n= 0$ is well-known. 
The case $n\neq 0$ follows from (1) and (2). 

(4) Combining (1), (2) and (3), we obtain the desired isomorphism as below 
\[
\begin{split}
\tuH^{n}(F \lotimes_{R} N) 
& \cong \tuH^{n}(F \lotimes_{R} \sigma^{\leq n}\sigma^{\geq n} N ) \\
& \cong \tuH^{n}(F \lotimes_{R}  \tuH^{n}(N)[-n] ) \\
& \cong \tuH^{0}(F \lotimes_{R} \tuH^{n}(N) ) \cong \tuH^{0}(F) \otimes_{H^{0}} \tuH^{n}(N).
\end{split}
\]

(5) The fist statement follows from (3). 
The second statement follows from (4).
\end{proof}

Using standard argument of triangulated categories,  
we obtain the following lemma.

\begin{lemma}\label{20171220037}
Let $a ,b \in \ZZ$ such that $a \leq b$. 
Then for any object $M \in \cF[a] \ast \cF[a + 1] \ast \cdots \ast \cF[b]$, 
we have $\fd M \leq b-a$.  
\end{lemma}

We give the definition of a sup-flat (spft) resolution of $M \in \sfD^{< \infty}(R)$.

\begin{definition}[spft morphism and spft resolution]\label{spft morphism definition} 
Let $M \in \sfD^{< \infty}(R), M \neq 0$. 

\begin{enumerate}[(1)]
\item 
A spft  morphism $f: F \to M $ is a morphism in $\sfD(R)$ 
such that $F \in \cF[- \sup M]$ 
and the morphism $\tuH^{\sup M}(f)$ is surjecitve. 

\item 
A spft morphism $f: F \to M $ is called minimal  if 
the morphism $\tuH^{\sup M}(f) $ is a flat cover.

\item 
A spft resolution $F_{\bullet}$ of $M$ is a sequence of exact triangles  for $i \geq  0$ 
with $M_{0} := M$ 
\[
M_{i+1} \xrightarrow{g_{i+1}} F_{  i } \xrightarrow{f_{i}} M_{i}
\]
such that $f_{i}$ is spft.

The following inequality folds
\[
\sup M_{i+1} = \sup F_{i+1} \leq \sup F_{i} =\sup M_{i}. 
\]

For a spft resolution $F_{\bullet}$ with the above notations, 
we set $\delta_{i} := g_{i-1} \circ f_{i}$. 
\[
\delta_{i} :F_{ i} \to F_{i-1}.
\]
Moreover we write 
\[
 \cdots \to F_{i} \xrightarrow{\delta_{i}} F_{ i -1} \to \cdots \to F_{1} \xrightarrow{\delta_{1}} F_{0} \to M. 
\]

\item 
A spft resolution $F_{\bullet}$ is said to have length $e$ if 
$F_{i } = 0$ for $i> e$ and 
$F_{e} \neq 0$.

\item 
A spft resolution $F_{\bullet}$ is called minimal if $f_{i}$ is minimal 
for $i \geq 0$.

\end{enumerate}
\end{definition}

It is clear that $\cP \subset \cF$. 
Therefore every $M \in \cD^{< \infty}(R)$  admits a spft morphism $F \to M$. 

We collect basic properties of spft resolutions. 

\begin{lemma}\label{201712192018}
Let $M \in \sfD^{< \infty}(R)$ and $f: F \to M$ a spft morphism 
and $N := \cone (f)[-1]$ the cocone of $f$. 
Assume that $1< \fd M $. 
Then,    $\fd N = \fd M -1 -\sup M +\sup N$. 
\end{lemma}

\begin{lemma}\label{spft resolution theorem lemma II}
Let $M \in \sfD^{< \infty}(R)$ and $d \in \NN$ a natural number. 
Then  
the following conditions are equivalent 

\begin{enumerate}[(1)]
\item 
$\fd  M  \leq  d$. 

\item 
For any spft resolution $F_{\bullet}$, 
there exists a natural number $e \in \NN$ 
such that 
$M_{e} \in \cF[-\sup F_{e}]$ 
and 
$ e+ \sup F_{0} -\sup M_{e}\leq d$. 
In particular, 
we have a spft resolution of length $e$. 
\[
 M_{e} \to F_{e-1} \to F_{e-2} \to \cdots F_{1} \to F_{0} \to M.
\]

\item 
$M$ has spft resolution $F_{\bullet}$ of length $e$ 
such that $ e + \sup F_{0} - \sup F_{e} \leq d$. 
\end{enumerate}
\end{lemma}

\begin{lemma}\label{spft resolution  lemma}
Let $L \in \Mod (H^{0})^{\op}, M \in \sfD^{< \infty}(R)$ and $F^{\bullet}$  a spft resolution of $M$. 
We denote the complexes below by $Y_{i}, Y'_{i}$. 
\[
\begin{split}
Y_{i} : 
\tuH^{\sup F_{i}}(F_{i+1}\lotimes_{R} L) \to 
\tuH^{\sup F_{i}}(F_{i}\lotimes_{R} L) \to
\tuH^{\sup F_{i}}(F_{i-1}\lotimes_{R} L)  \\ 
Y'_{i} : 
\tuH^{\sup F_{i}}(F_{i+1}) \otimes_{H^{0}} L \to
\tuH^{\sup F_{i}}(F_{i}) \otimes_{H^{0}} L \to
\tuH^{\sup F_{i}}(F_{i-1}) \otimes_{H^{0}} L 
\end{split}
\] 
Then, 
\[
\tuH^{n}(M \lotimes_{R} L) =
\begin{cases} 
0 & n \neq -i +\sup F_{i} \textup{ for any } i \geq 0 \\
\tuH(Y_{i}) & n = -i +\sup F_{i} \textup{ for some } i \geq 0 
\end{cases}
\]

Moreover, in the case $n =- i + \sup F_{i}$,  we have  
\[
\tuH(Y_{i}) =
\begin{cases}
\tuH^{\sup F_{i}}(F_{i}) \otimes_{H^{0}} L  & \sup F_{i -1} \neq \sup F_{i} \neq \sup F_{i+1}, \\
\Coker[\tuH^{\sup F_{i}}(F_{i+1}) \otimes_{H^{0}} L \to
\tuH^{\sup F_{i}}(F_{i}) \otimes_{H^{0}} L ],  &
 \sup F_{i-1} \neq \sup F_{i} = \sup F_{i+1}, \\ 
 \Ker[
\tuH^{\sup F_{i}}(F_{i}) \otimes_{H^{0}} L \to
\tuH^{\sup F_{i}}(F_{i-1}) \otimes_{H^{0}} L 
 ], & 
 \sup F_{i -1} = \sup F_{i} \neq \sup F_{i+1}, \\ 
\tuH(Y'_{i}), & 
 \sup F_{i -1} = \sup F_{i} = \sup F_{i +1}. 
\end{cases}  
\]
\end{lemma}

We recall the notion of pure-injection. 
Let $A$ be an algebra. 
A morphism $f: M \to N$ between $A$-modules said to be a  pure-injection 
if $f \otimes_{A} N: M \otimes_{A} N \to N \otimes_{A} N$ is injective for all $A^{\op}$-modules $N$.

\begin{theorem}\label{spft resolution theorem}
Let $M \in \sfD^{< \infty}(R)$ and $d \in \NN$ a natural number. 
Then  
the following conditions are equivalent 

\begin{enumerate}[(1)]
\item 
$\fd  M  = d$.

\item 
For any spft resolution $F_{\bullet}$, 
there exists a natural number $e \in \NN$ 
which satisfying the following properties 

\begin{enumerate}[(a)]
\item $ M_{e} \in \cF[-\sup M_{e}]$. 

\item 
$d = e+ \sup F_{0} -\sup M_{e}$. 

\item 
$\tuH^{\sup M_{e}}(g_{e})$ is not a pure-injection. 
\end{enumerate}

\item 
$M$ has spft resolution $F_{\bullet}$ of length $e$ 
which satisfies the following properties.

\begin{enumerate}[(a)]

\item 
$d = e+ \sup F_{0} -\sup F_{e}$. 

\item 
$\tuH^{\sup P_{e}}(\delta_{e})$ is not a pure-injection. 
\end{enumerate} 
\end{enumerate}
\end{theorem}

As is the same with the sppj resolution, 
Theorem \ref{spft resolution theorem} has the following consequences. 
We denote by $\sfD(R)_{\textup{ffd}}$ the full subcategory consisting of $M$ having finite flat dimension. 

\begin{proposition}
\[
\sfD(R)_{\textup{ffd}} = \thick \cF = \bigcup \cF[a]\ast \cF[a +1] \ast \cdots \ast \cF[b]. 
\]
where $a,b$ run all the pairs of integers such that $a \leq b$. 
\end{proposition}

\subsection{Conjecture about an explicit description of $\cF$.}

The following is a spft version of Lemma \ref{20171219205II} 
is still a conjecture. 

We denote by  $\cF' \subset \sfD(R)$  
the full subcategory of those object $F' \in \sfD(R)$ 
which is a quasi-isomorphism class of DG-$R$-modules 
of the forms $F^{0}\otimes_{R^{0}}R$ 
for some flat $R^{0}$-module $F^{0}$.
It is easy to check that $\cF' \subset \cF$

\begin{conjecture}\label{201712192021} 
$\cF' = \cF$. 
\end{conjecture}




\section{The global dimension}\label{The global dimension}

We introduce the notion of the global dimension of a connective DG-algebra $R$.

\begin{theorem}\label{global dimension theorem}
Let $R$ be a connective DG-algebra. Then the following numbers are the same. 
\[
\begin{split}
&\sup\{ \pd M-\champ M \mid M \in \sfD^{< \infty}(R) \} \\
& \sup\{ \pd M \mid  M \in \Mod H^{0} \} \\
&\sup\{  \injdim M - \champ M \mid M \in \sfD^{> -\infty}(R) \} \\
&\sup\{\injdim M\mid M \in \Mod H^{0}\} 
\end{split}
\]
This common number is called the \emph{(right) global dimension} of $R$ and is denoted as $\gldim R$. 
\end{theorem}

\begin{proof}
For simplicity, we set the $i$-th value in question to be $v_{i}$ for $i = 1,2,3,4$. 
For example $v_{2} := \inf\{ n \in \NN \mid \pd M \leq n \textup{ for all } M \in \Mod H^{0} \}$.

We claim   $v_{1} =v_{2}$. 
It is clear $v_{2} \leq v_{1}$. 
Therefore it is enough to prove  that 
  $\pd M \leq v_{2} + \champ M$ for any $M \in \sfD^{< \infty}$.   
In the case where  $\champ M = \infty$, there is nothing to prove. 
We deal with the case $\champ M < \infty$ by the induction on $a=\champ M$. 
The case $\champ M = 0$ is obvious. 
Assume that the case $< a$ is already proved. 
Let 
$
\sigma^{< \sup M} M \to M \to \tuH^{\sup} (M)[-\sup M] \to 
$ be the exact triangle induced from the standard $t$-structure. 
Note that  $\champ \sigma^{< \sup M}M < \champ M$. 
Using  Lemma \ref{201804030002} and the induction hypothesis, 
we deduce the desired inequality. 

We can prove $v_{3} =v_{4}$ in the same way.  

We prove $v_{2} \leq v_{4}$. If $v_{4} = \infty$, then there is nothing to prove. 
Assume that $v_{4} < \infty$.  
Let $M \in \Mod H^{0}$. 
Then for any $N \in \Mod H^{0}$, 
we have $\RHom(M,N) \in \sfD^{[0,v_{4}]}$. 
Thus, by Theorem \ref{sppj resolution theorem lemma II}, we deduce that $\pd M \leq v_{4}$. 
This proves the desired inequality. 

We can show $v_{4} \leq v_{2}$ in the same way and  
prove the equality $v_{2} = v_{4}$. 
\end{proof}

\begin{remark}
For any connective DG-algebra $R$,  we have 
\[ \sup\{ \pd M\mid M \in \sfD^{< \infty}(R) \} = \infty,\]
since $\pd (R \oplus R[n]) = n$ for $n \in \NN$. 
\end{remark}

Observe that if $R$ is an ordinary algebra, 
then the global dimension defined in Theorem \ref{global dimension theorem} 
coincides with the ordinary global dimension. 
The ordinary global dimensions is not preserved by derived equivalence, 
but their finiteness is preserved.
We prove the DG-version. 

Let $R$ and $S$ be connective DG-algebra. 
Assume that they are derived equivalent to each other.
Namely, there exists an equivalence $\sfD(R) \simeq \sfD(S)$ of triangulated categories, 
by which we identify $\sfD(R)$ with $\sfD(S)$. 

\begin{proposition}
Under the above situation the following assertions hold.
\begin{enumerate}[(1)]
\item  $\pd_{S} R < \infty$. 

\item $\gldim S \leq \gldim R + \pd_{S} R$. 

\item $gldim R < \infty$ if and only if $\gldim S < \infty$. 
\end{enumerate}
\end{proposition}

\begin{proof}
(1) It is well-known that $R \in \thick S$. 
Thus, in particular $\pd_{S}R < \infty$.

(2) 
If $\gldim R = \infty$, then there is nothing to prove. 
We assume $\gldim R < \infty$. 
Let $M \in \Mod \tuH^{0}(S)$. 
Then, $\RHom(R, M) \in \sfD^{[-\sup_{S} R,\pd_{S}R -\sup_{S}R]}(\kk)$. 
Therefore $M$ belongs to $D^{[-\sup_{S} R,\pd_{S}R -\sup_{S}R]}(R)$. 
Thus we have $\champ_{R} M \leq \pd_{S} R < \infty$ and 
\begin{equation}\label{201804051540}
\champ_{R} M - \sup_{R}M -\sup_{S} R= - \inf_{R} M -\sup_{S} R \leq 0.
\end{equation}

Note that  $\pd_{R} M \leq \gldim R + \champ_{R} M < \infty$. 
Let $P_{\bullet}$ be  a sppj resolution of $M$ as an object of $\sfD(R)$ 
of length $e$ such that $\pd_{R} M = e + \sup_{R} M - \sup_{R} P_{e}$.
 Then, $M \in \cP[-\sup_{R} P_{0}] \ast \cdots \ast \cP[e - \sup_{R} P_{e}]$. 
 Since for $P\in \cP$ we have $\pd_{S} P = \pd_{S} R$, 
 using Corollary \ref{201804030002}, we deduce the  inequality  below.  
\begin{equation}\label{201804051246} 
\pd_{S} M  
 \leq  \pd_{S} R + e -\sup_{S} P_{e}  =\pd_{S} R +\pd_{R} M  -\sup_{R} M  +\sup_{R} P_{e} -\sup_{S} P_{e}.
 \end{equation}

Observe that $\sup_{R} P_{e} - \sup_{S} P_{e} = -\sup_{S} R$.
Thus, combining the inequalities \eqref{201804051246}, \eqref{201804051540} 
and $\pd_{R} M \leq \gldim R + \champ_{R} M$, 
we obtain the desired inequality. 
\[
\begin{split}
\pd_{S} M  
& \leq  \pd_{S} R +  \pd_{R} M  -\sup_{R} M  +\sup_{R} P_{e} -\sup_{S} P_{e} \\
& \leq \pd_{S} R + \gldim R + \champ_{R}M -\sup_{R} M -\sup_{S}R\\
& \leq \pd_{S} R + \gldim R
\end{split}
\]

(3) follows from (1) and (2). 
\end{proof}

To finish this section we identify a connective DG-algebra of global dimension $0$.

\begin{proposition}\label{global dimension 0}
For a connective DG-algebra $R$, the following conditions are equivalent. 
\begin{enumerate}[(1)]
\item $\gldim R= 0$. 

\item $R$ is an ordinary algebra (i.e., $H^{<0}= 0$)  which is semi-simple.
\end{enumerate}
\end{proposition}

\begin{proof}
The implication (2) $\Rightarrow$ (1) is clear. 
We prove the implication (1) $\Rightarrow$ (2). 
We only have to show that $H^{< 0} = 0$. 
Since $\pd H^{0} = 0$, the canonical morphism $R \to H^{0}$ which is sppj splits in $\sfD(R)$. 
Thus, the canonical exact sequence $ 0 \to H^{< 0} \to H \to H^{0} \to 0$ of graded $H$-modules splits. 
It follows that $H^{< 0} = 0$. 
\end{proof}



\end{document}